\documentclass[a4paper,12pt]{amsart}

\usepackage{amsmath,amsthm,amssymb,amsfonts,amscd,bm,mathtools}

\usepackage{float, graphicx, mdframed, tikz, verbatim}
\usepackage[margin=1in]{geometry}

\usepackage{enumerate}

\usepackage[hidelinks]{hyperref}

\theoremstyle{plain}
\newtheorem{theorem}{Theorem}[section]
\newtheorem{conjecture}[theorem]{Conjecture}
\newtheorem{corollary}[theorem]{Corollary}
\newtheorem{lemma}[theorem]{Lemma}
\newtheorem{proposition}[theorem]{Proposition}

\theoremstyle{definition}
\newtheorem{definition}[theorem]{Definition}
\newtheorem{note}[theorem]{Note}

\theoremstyle{remark}
\newtheorem{example}[theorem]{Example}
\newtheorem{remark}[theorem]{Remark}

\DeclareMathOperator{\NN}{\mathbb{N}}
\DeclareMathOperator{\ZZ}{\mathbb{Z}}

\DeclareMathOperator{\RR}{\mathbb{R}}
\DeclareMathOperator{\CC}{\mathbb{C}}

\DeclareMathOperator{\PP}{\mathbb{P}}

\DeclareMathOperator{\cdeg}{cdeg}
\DeclareMathOperator{\comaj}{comaj}
\DeclareMathOperator{\defining}{def}
\DeclareMathOperator{\des}{des}
\DeclareMathOperator{\fdes}{fdes}
\DeclareMathOperator{\fixed}{fixed}
\DeclareMathOperator{\fmaj}{fmaj}
\DeclareMathOperator{\Frob}{Frob}
\DeclareMathOperator{\fxd}{fxd}
\DeclareMathOperator{\GL}{GL}

\DeclareMathOperator{\maj}{maj}

\DeclareMathOperator{\ndes}{ndes}
\DeclareMathOperator{\nmaj}{nmaj}
\DeclareMathOperator{\pdeg}{pdeg}
\DeclareMathOperator{\Proj}{Proj}
\DeclareMathOperator{\RHT}{RHT}
\DeclareMathOperator{\RST}{RST}
\DeclareMathOperator{\sign}{sign}

\DeclareMathOperator{\SSYT}{SSYT}
\DeclareMathOperator{\stand}{stand}
\DeclareMathOperator{\Symm}{Symm}
\DeclareMathOperator{\SYT}{SYT}
\DeclareMathOperator{\tr}{tr}
\DeclareMathOperator{\triv}{triv}
\DeclareMathOperator{\weight}{weight}
\DeclareMathOperator{\wcomaj}{wcomaj}
\DeclareMathOperator{\wdes}{wdes}
\DeclareMathOperator{\wDes}{Des}

\newcommand{\qbinom}{\genfrac{[}{]}{0pt}{}}

\makeatletter
\def\@tocline#1#2#3#4#5#6#7{\relax
  \ifnum #1>\c@tocdepth 
  \else
    \par \addpenalty\@secpenalty\addvspace{#2}%
    \begingroup \hyphenpenalty\@M
    \@ifempty{#4}{%
      \@tempdima\csname r@tocindent\number#1\endcsname\relax
    }{%
      \@tempdima#4\relax
    }%
    \parindent\z@ \leftskip#3\relax \advance\leftskip\@tempdima\relax
    \rightskip\@pnumwidth plus4em \parfillskip-\@pnumwidth
    #5\leavevmode\hskip-\@tempdima
      \ifcase #1
       \or\or \hskip 1em \or \hskip 2em \else \hskip 3em \fi%
      #6\nobreak\relax
    \dotfill\hbox to\@pnumwidth{\@tocpagenum{#7}}\par
    \nobreak
    \endgroup
  \fi}
\makeatother
\newcommand\partitionfr[1]{
	\coordinate (prev) at (0,0);
	\foreach \dir in {#1}{
		\draw[color = black!60, thick] (prev) -- +(0,1) coordinate (prev);
		\draw[color = black!60, thick] (prev)+(0,-1) grid +(\dir,0);
	};
}

\begin{document}

\title[Color rules for cyclic wreath products and semigroup algebras]{Color rules for cyclic wreath products and semigroup algebras from projective toric varieties}
\author{Fabián Levicán-Santibáñez}
\author{Marino Romero}
\address{Fakultät für Mathematik, Universität Wien, Austria}
\email{fabian.levican@univie.ac.at}
\email{marino.romero@univie.ac.at}
\date{August 5, 2025}

\begin{abstract}
We introduce the notion of ``color rules'' for computing class functions of $Z_k \wr S_n$, where $Z_k$ is the cyclic group of order $k$ and $S_n$ is the symmetric group on $n$ letters.  
Using a general sign-reversing involution and a map of order $k$, we give a combinatorial proof that the irreducible decomposition of these class functions is given by a weighted sum over semistandard tableaux in the colors.
Since using two colors at once is also a color rule, we are consequently able to decompose arbitrary tensor products of representations whose characters can be computed via color rules. This method extends to class functions of $G\wr S_n$ where $G$ is a finite abelian group. 
We give a number of applications, including decomposing tensor powers of the defining representation, along with a combinatorial proof of the Murnaghan--Nakayama rule for $Z_k \wr S_n$.

Our main application is to the study of the linear action of $Z_k \wr S_n$ on bigraded affine semigroup algebras arising from the product of projective toric varieties.
In the case of the product of projective spaces, our methods give the decomposition of these bigraded characters into irreducible characters, thus deriving equivariant generalizations of Euler-Mahonian identities.
\end{abstract}

\maketitle

\tableofcontents

\section{Introduction}

We will present a general method for computing irreducible decompositions of class functions of $G \wr S_n$ for when $G$ is an abelian group and $S_n$ is the symmetric group on $n$ letters (though we will focus on the case when $G = Z_k$ is the cyclic group of order $k$, from which the more general case follows). Our proof is purely combinatorial and noteworthy for a number of reasons. \\

First, the family of class functions we will be interested in are attained by ``coloring'' cycles according to some ``color rule'' (see Section~\ref{section:color-rules}), and any character can be written as a sum of these types of class functions (see Subsection~\ref{subsection:basis}). 
Given such a class function, we will show that the multiplicities of irreducible characters can be enumerated by weighted sums over certain semistandard tableaux in the colors. This will be our main result, Theorem \ref{theorem:mainresult}. 

There are a number of consequences from this perspective. For one, tensor products of representations whose characters are computed via color rules can be immediately decomposed into irreducible representations, since, conveniently, having two simultaneous coloring rules is itself a single coloring rule! This is explained by Proposition~\ref{proposition:products_of_colors}.
As a consequence, we get a large number of applications, including the decomposition of tensor powers of the defining representation of $Z_k \wr S_n$.
Furthermore, the class functions are not necessarily characters of a representation and can include weights.
For example,
in Subsection~\ref{section:murnaghan-nakayama}, we  get a combinatorial proof for the Murnaghan-Nakayama rule for $Z_k \wr S_n$ written in terms of multi-Schur functions.
These methods also restrict to consequences for the symmetric group, as we will show in Subsection~\ref{subsection:symmetric_group}.\\ 

Second, our proof uses a remarkably general weight-preserving, sign-reversing involution method.
The core idea behind this method stems from a special case of our setting, found in \cite{MendesRomero}, where a sign-reversing involution is used to decompose tensor powers of the defining representation of $S_n$. This idea was also later used in \cite{Mendes} to give a combinatorial proof of the Murnaghan-Nakayama rule for $S_n$. As previously mentioned, our framework specializes to these cases. Most importantly, our point of view gives a unified combinatorial framework for understanding arbitrary products of these types of class functions. \\

A key application of our methods is to the study of certain bigraded group characters arising from actions on affine semigroup algebras, $K[P^{\times n}]$, coming from the product of projective toric varieties. We will study a bigrading introduced by Chapoton in \cite{chapoton-analogues}, which is similar to (but distinct from) the one introduced by Reiner and Rhoades in \cite{reiner-rhoades-harmonics}. Then, we will endow $K[P^{\times n}]$ with the structure of a bigraded (or multigraded) $Z_k \wr S_n$-module. In particular, we will compute the character of $K[P^{\times n}]$ and give a combinatorial rule for computing its decomposition into irreducible representations (see Theorem \ref{thm:projectivecharacter}). 
For $k=1$, we will also describe a link to Stapledon's \emph{equivariant Ehrhart theory} \cite{stapledon-equivariant} by introducing a new refinement using this bigrading.

An interesting case is when $P= \Delta_1 = [0,1]$ is the unit interval. A regular sequence in $K[(\Delta_1)^{\times n}]$ is given in \cite{adeyemo-szendroi-refined} (for $k=1$) and \cite{braun-olsen-statistics} (for any $k$). In \cite{braun-olsen-statistics}, they interpret the Hilbert series of quotients of $K[(\Delta_1)^{\times n}]$ by these regular sequences as \emph{Euler-Mahonian identities}, which have appeared in various contexts within algebraic combinatorics, as discussed after their Theorem 1.1. 

As a special case of our work, we will study the products of unit $m$-simplices $P = \Delta_m$, corresponding geometrically to the product of projective spaces with the Segre embedding. We will give a regular sequence in $K[(\Delta_m)^{\times n}]$ for $k=1$, conjecture a regular sequence for every $k$, and therefore give an expression for the bigraded character of the quotients of $K[(\Delta_m)^{\times n}]$ by these regular sequences.
In particular, for $m=1$ and any $k$ and $n$, our methods give an equivariant generalization of these aforementioned Euler-Mahonian identities (see Table~\ref{tab:equivariant-euler-mahonian-identities}). 
\\

Our work is organized in the following way. Section \ref{section:Preliminaries} introduces the main combinatorial structures, along with the representation theory of $Z_k \wr S_n$. Section \ref{section:color-rules} will introduce the definition of a ``color rule'' and state the main combinatorial result. In Section~\ref{section:some_examples}, we give a list of examples and applications from this point of view.
This will be followed by Section~\ref{section:semigroup-algebras}, where we introduce affine semigroup algebras and projective toric varieties, describe the action of $Z_k \wr S_n$, and state our results about the product of projective spaces. We also discuss links to equivariant Ehrhart theory and to Euler-Mahonian identities.
Lastly, Section \ref{section:main_proof} gives a proof of our main theorem regarding color rules.

\section{Preliminaries} \label{section:Preliminaries}

\subsection{Partitions}

A partition $\lambda= (\lambda_1 , \cdots ,\lambda_{\ell}) \vdash n$ is a sequence of non-increasing positive integers whose total sum is $|\lambda|= n$ and length is $\ell(\lambda) = \ell$. We represent partitions by its diagram in the French convention, see Figure~\ref{fig:partitions}. 

We will write $\vec{\lambda} = (\lambda^0, \lambda^1, \dots, \lambda^{k-1}) \vdash_k n$ if $\vec{\lambda}$ is a sequence of $k$ partitions such that $|\lambda^0| + |\lambda^1| +\cdots + |\lambda^{k-1}| = n$. Note that some of these partitions may be empty. We give this sequence a Young diagram interpretation by drawing successive partitions corner to corner. More specifically, the south-east corner of $\lambda^i$ is connected to the north-west corner of $\lambda^{i+1}$. If $\lambda^i$ is an empty partition, we let $\lambda^i$ be a cell with $\emptyset$ written within, and we proceed with the concatenation. Similarly to ordinary integer partitions, we denote its size by $|\vec{\lambda}| = n$.
For example, the sequence $\vec{\lambda} = ((3,1),(\emptyset),(2,2))$ depicted on the right-hand side of Figure~\ref{fig:partitions} has size $8$. It is important that we leave the empty square empty, and any interactions with cells of $\vec{\lambda}$ should avoid this square. 
\\

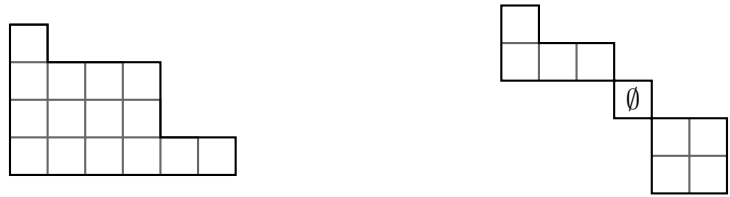
\begin{figure}[h]
\begin{align*}\vcenter{\hbox{
\begin{tikzpicture}[scale = 1/2]
\draw [color=black!60, thick] (5,0) -- (5,1)-- (0,1);
\draw [color=black!60, thick] (4,0) -- (4,2)-- (0,2);
\draw [color=black!60, thick] (3,0) -- (3,3)-- (0,3);
\draw [color=black!60, thick] (2,0) -- (2,3);
\draw [color=black!60, thick] (1,0) -- (1,4)-- (0,4);
\draw [color=black, thick] (0,0) -- (6,0)  -- (6,1) -- (4,1) -- (4,3) -- ( 1,3) -- (1,4) -- (0,4) -- cycle;
\end{tikzpicture}}}
&&\vcenter{\hbox{
 \begin{tikzpicture}[scale = 1/2, baseline = 10ex]
 \draw [color = black!60, thick] (0,4) -- (1,4);
\draw [color = black!60, thick] (1,3) -- (1,4);
\draw [color = black!60, thick] (2,3) -- (2,4);
\draw [color = black!60, thick] (4,1) -- (6,1);
\draw [color = black!60, thick] (5,0) -- (5,2);
\node at (3.5,2.5) {$\emptyset$};
\draw [black, thick] (0,3) -- (3,3)  -- (3,2)  -- (4,2) -- (4, 0) -- (6,0) -- (6,2) --(4,2) -- (4,3) -- (3,3) -- (3,4) -- (1,4) -- (1,5) -- (0,5) -- cycle;
\end{tikzpicture}}}
\end{align*}
\caption{The partition $\lambda= (6,4,4,2)$ is drawn on the left; and on the right is the sequence of partitions $\vec{\lambda} = ((3,1),\emptyset,(2,2)) \vdash_3 8$.}
\label{fig:partitions}
\end{figure}

Suppose we are given a multiset $F=[f_1^{m_1}, f_2^{m_2},\dots]$, and suppose that we fix an order on the elements of $F$, say $f_1<f_2<\cdots$ (the choice generally does not matter). A semistandard tableau $T\in \SSYT(\vec{\gamma},F)$ of shape $\vec{\gamma}$ and content $F$ is a filling of the cells of $\vec{\gamma}$ so that every cell gets one element from $F$; $f_i$ appears at most $m_i$ times; the rows are weakly increasing when read from left to right; and columns are strictly increasing from bottom to top. See Figure~\ref{fig:SSYT_and_SYT}.

The set of standard Young tableaux of shape $\vec{\gamma}$, $\SYT(\vec{\gamma}) \coloneq \SSYT(\vec{\gamma},[1,2,\dots,|\vec{\gamma}|])$, consists of semistandard Young tableaux with entries in $\{1,\dots, |\vec{\gamma}|\}$ so that each label appears exactly once. 
Every element $T \in \SSYT(\vec{\gamma},F)$ can be ``standardized'' to an element $\stand(T) \in \SYT(\vec{\gamma})$ where we replace the entries by their index in the reading order. That is, if there are $a_i$ cells labeled $f_i$, then replace the $a_1$ cells containing $f_1$ by $1,\dots, a_1$, increasing from left to right, replace the $a_2$ cells containing $f_2$ with $a_1+1,\dots, a_1+a_2$, increasing from left to right, and so on. For an example, see Figure~\ref{fig:SSYT_and_SYT}.

\begin{figure}
\begin{align*}
\vcenter{\hbox{
 \begin{tikzpicture}[scale = 1/2, baseline = 10ex]
 \draw [color = black!60, thick] (0,4) -- (1,4);
\draw [color = black!60, thick] (1,3) -- (1,4);
\draw [color = black!60, thick] (2,3) -- (2,4);
\draw [color = black!60, thick] (4,1) -- (6,1);
\draw [color = black!60, thick] (5,0) -- (5,2);
\draw [black, thick] (0,3) -- (3,3)  -- (3,2)  -- (4,2) -- (4, 0) -- (6,0) -- (6,2) --(4,2) -- (4,3) -- (3,3) -- (3,4) -- (1,4) -- (1,5) -- (0,5) -- cycle;
\node at (3.5,2.5) {$\emptyset$};
\node at (-1.5,2.5) {$T=$};
\node at (.5,3.5) {$f_1$};
\node at (1.5,3.5) {$f_1$};
\node at (4.5,.5) {$f_1$};
\node at (4.5,1.5) {$f_2$};
\node at (.5,4.5) {$f_3$};
\node at (2.5,3.5) {$f_3$};
\node at (5.5,.5) {$f_3$};
\node at (5.5,1.5) {$f_4$};
\end{tikzpicture}}}
&&
\vcenter{\hbox{
 \begin{tikzpicture}[scale = 1/2, baseline = 10ex]
 \draw [color = black!60, thick] (0,4) -- (1,4);
\draw [color = black!60, thick] (1,3) -- (1,4);
\draw [color = black!60, thick] (2,3) -- (2,4);
\draw [color = black!60, thick] (4,1) -- (6,1);
\draw [color = black!60, thick] (5,0) -- (5,2);
\draw [black, thick] (0,3) -- (3,3)  -- (3,2)  -- (4,2) -- (4, 0) -- (6,0) -- (6,2) --(4,2) -- (4,3) -- (3,3) -- (3,4) -- (1,4) -- (1,5) -- (0,5) -- cycle;
\node at (3.5,2.5) {$\emptyset$};
\node at (-2.5,2.5) {$\stand(T)=$};
\node at (.5,3.5) {$1$};
\node at (1.5,3.5) {$2$};
\node at (4.5,.5) {$3$};
\node at (4.5,1.5) {$4$};
\node at (.5,4.5) {$5$};
\node at (2.5,3.5) {$6$};
\node at (5.5,.5) {$7$};
\node at (5.5,1.5) {$8$};
\end{tikzpicture}}}
\end{align*}
\caption{On the left, we have a semistandard tableau $T \in \SSYT( ((3,1),\emptyset,(2,2)), [f_1^3,f_2,f_3^6,f_4])$. On the right, we have its standardization $\stand(T) \in \SYT(((3,1),\emptyset,(2,2)))$.
}
\label{fig:SSYT_and_SYT}
\end{figure}
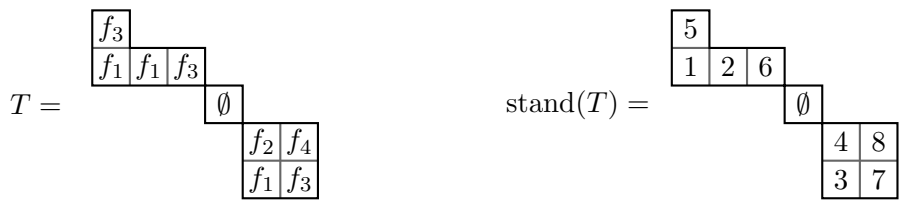

\subsection{Wreath products}
For references regarding wreath products with the symmetric group and connections to symmetric functions, see \cite{MRW} and \cite{Macdonald}. \\

For the remainder of the paper, let $n$ and $k$ be fixed positive integers. Let $Z_k$ be the cyclic group consisting of the $k^{\text{th}}$ roots of unity $\{u_0,u_1,\dots,u_{k-1}\}$, generated by $u_1$ (so that $u_r = u_1^r$); and let $S_n$ be the symmetric group on the letters $\{1,2,\dots, n\}$. \\

Let $Z_k \wr S_n$ denote the group formed by the set of $n$ by $n$ matrices whose rows and columns have only one nonzero entry coming from $Z_k$. Since this group is introduced in matrix form, the elements $\sigma$ of $Z_k \wr S_n$ are given a natural ``defining'' representation, which is the matrix $\Pi_{\defining}(\sigma)$ from their definition.
For instance, one such matrix in $Z_4 \wr S_6$ is given by
\begin{equation*} \Pi_{\defining}(\tau) = 
 \begin{pmatrix}
 u_0 & 0 & 0 & 0 & 0 & 0 \\
 0 & 0 & u_3 & 0 & 0 & 0 \\
 0 & 0 & 0 & 0 & u_0 & 0 \\
 0 & 0 & 0 & u_1 & 0 & 0 \\
 0 & u_2 & 0 & 0 & 0 & 0 \\
 0 & 0 & 0 & 0 & 0 & u_3 \\  
 \end{pmatrix}.
\end{equation*}

 We may also consider the elements in $Z_k \wr S_n$ as weighted permutations. That is $\sigma \in Z_k \wr S_n$ can be written in two-line notation as
 \begin{equation*}
 \sigma = \left(\begin{array}{@{\,}c@{\ }c@{\ }c@{\ }c@{\ }}
1&2&\cdots&n\\
u_{a_1} \sigma_1&u_{a_2}\sigma_2&\cdots&u_{a_n}\sigma_n
\end{array}\right),
\end{equation*}
where $u(\sigma, \sigma_i) \coloneq u_{a_i}$ is some $k^{\text{th}}$ root of unity and $\sigma_1 \cdots \sigma_n$ is a permutation in $S_n$. We will then say that $i$ goes to $u_{a_i} \sigma_i$.
We will often refer to the numbers $1,\dots,n$ appearing in $\sigma_1\cdots \sigma_n$ as indices, so that the index $i$ appears with the root of unity $u(\sigma,i).$
In terms of the representation $\Pi_{\defining}(\sigma)$, $u(\sigma,i)$ is the root of unity in row $i$.
\\

We will use a particular cycle notation. Given $\sigma \in Z_k \wr S_n$, we can write $\sigma$ uniquely as a product of disjoint cycles in the following way:  $u(\sigma, i)~ i$ precedes $u(\sigma,\sigma_i)~ \sigma_i$ in a cycle; every cycle begins with the smallest index $i$; and arrange the cycles so that they are decreasing from left to right according to the smallest index in each cycle. 
For example, the element $\tau$ above can be written as
 \begin{equation*}
 \tau = 
 \left(\begin{array}{@{\,}c@{\ }c@{\ }c@{\ }c@{\ }c@{\ }c@{\ }}
1&2&3&4&5&6 \\
u_0 1 & u_2 5 & u_3 2& u_1 4 & u_0 3 & u_3 6
\end{array}\right) = (u_3 6)(u_1 4)(u_3 2, u_2 5 , u_0 3)(u_0 1) \in Z_4 \wr S_6.
 \end{equation*}
Because we are writing the first indices of each cycle in decreasing order, we will call this unique arrangement the decreasing cycle notation of $\sigma$. 
\begin{remark}\label{remark:decreasing_cycle_bijection}
Given $\sigma$ in decreasing cycle notation, one can erase the parentheses and get a bijection from $Z_k \wr S_n$ to itself, where the above permutation in cycle notation would be sent to $ u_3 6~u_1 4~u_3 2~u_2 5 ~u_0 3~u_0 1$, written in one-line notation. One gets the original cycle notation by taking a left-to-right minimum as the location of a new cycle.
\end{remark}

Given a cycle $c = (u_{a_1} i_1, \dots, u_{a_m} i_m)$, we say that $c$ is a $C_{j}$-cycle if $u_{a_1} \cdots u_{a_m} = u_j$. This $u_j$ will also be called the C-type of the cycle $c$. 
We may also say that $c$ is an $m$-cycle, in reference to its length; so $1$-cycles are cycles of length $1$, whereas $C_0$-cycles are cycles of C-type $u_0=1$.
The conjugacy classes are indexed by sequences of partitions $\vec{\lambda} \vdash_k n$:
\begin{equation*}
C_{\vec{\lambda}} = \{\sigma \in C_k \wr S_n :
 \text{the $C_j$-cycles in $\sigma$ are of lengths $\lambda^j_1, \dots, \lambda^j_{\ell(\lambda^j)}$ for } j=0, \dots, k-1 \}.
\end{equation*}

From the previous example $\tau \in Z_4 \wr S_6$, since $u_3 = u_3$, $u_1 = u_1$, $u_3 u_2 u_0 = u_1$, and $u_0 = u_0$, 
\begin{align*}
(u_0 1) && \text{is a $C_{0}$-cycle}, \\
(u_3 2, u_2 5 , u_0 3) \ \text{ and } \ (u_1 4) && \text{are $C_{1}$-cycles, and }\\
(u_3 6) && \text{is a $C_{3}$-cycle}.
\end{align*}
Therefore, $\tau$ is in the conjugacy class described by the sequence of partitions \[
((1),(3,1),(\emptyset),(1)) \vdash_4 6.
\]
\subsection{Irreducible characters}
The irreducible characters of $Z_k$ are given by the functions $\{\Psi^j : 0 \leq j < k\}$ defined by setting $\Psi^j(u)= u^j$. In particular, these characters are themselves representations and multiplicative: $\Psi^j(uv) = \Psi^j(u) \Psi^j(v)$. This will be the property we need to prove Theorem \ref{theorem:mainresult}, which is why we are able to generalize the main result to wreath products with abelian groups. \\

For each $\vec{\gamma} \vdash_k n$ there is a corresponding irreducible character of $Z_k \wr S_n$, $\chi^{\vec{\gamma}}$, whose value at $\sigma$ can be computed by a Murnaghan-Nakayama rule, as shown in \cite{MRW}. This is what we now describe. \\

For a partition $\lambda$, a \emph{rim hook} is a contiguous path along the cells of the north-east boundary of $\lambda$ consisting of East and South steps whose removal leaves a partition. 
We insert the rim hook by drawing a line segment through the sequence of cells. And for a rim hook $\zeta$, we define
\begin{equation*}
\sign(\zeta) \coloneq (-1)^{\text{number of South steps in $\zeta$}}.
\end{equation*}
For example, Figure~\ref{fig:rimhooks} depicts a rim hook on the left; and the second image is not a rim hook since the removal of the sequence of cells does not leave a partition. 
Since the first shape has a rim hook that traverses $2$ horizontal lines, its sign is given by $(-1)(-1) = +1$. \\

\begin{figure}
\begin{align*}
& \begin{tikzpicture}[scale = 1/2]
\draw [color=black!60, thick] (5,0) -- (5,1)-- (0,1);
\draw [color=black!60, thick] (4,0) -- (4,2)-- (0,2);
\draw [color=black!60, thick] (3,0) -- (3,3)-- (0,3);
\draw [color=black!60, thick] (2,0) -- (2,3);
\draw [color=black!60, thick] (1,0) -- (1,4)-- (0,4);
\draw [color=black!75, thick] (1.5,2.5) -- (3.5,2.5)-- (3.5,.5) -- (5.5,.5);
\draw [color=black, thick] (0,0) -- (6,0)  -- (6,1) -- (4,1) -- (4,3) -- ( 1,3) -- (1,4) -- (0,4) -- cycle;
\fill[color = black, thick] (1.5,2.5) circle (.5ex);
\fill[color = black, thick] (5.5,.5) circle (.5ex);
\end{tikzpicture}
& \begin{tikzpicture}[scale = 1/2]
\draw [color=black!60, thick] (5,0) -- (5,1)-- (0,1);
\draw [color=black!60, thick] (4,0) -- (4,2)-- (0,2);
\draw [color=black!60, thick] (3,0) -- (3,3)-- (0,3);
\draw [color=black!60, thick] (2,0) -- (2,3);
\draw [color=black!60, thick] (1,0) -- (1,4)-- (0,4);
\draw [color=black!75, thick] (1.5,2.5) -- (3.5,2.5)-- (3.5,.5);
\draw [color=black, thick] (0,0) -- (6,0)  -- (6,1) -- (4,1) -- (4,3) -- ( 1,3) -- (1,4) -- (0,4) -- cycle;
\fill[color = black, thick] (1.5,2.5) circle (.5ex);
\fill[color = black, thick] (3.5,.5) circle (.5ex);
\end{tikzpicture}
\end{align*}
\caption{The path depicted on the left is a rim hook with positive sign, while the one on the right-hand side is not.}
\label{fig:rimhooks}
\end{figure}
 
Let $c = (c_1,\dots, c_m)$ be a composition of $n$, and let $\vec{\lambda}\vdash_k n$. Also, let $a=(a_1, \dots, a_m)$ be a sequence of elements in $Z_k$. A \emph{rim hook tableau} $T\in \RHT_{\vec{\lambda},(c,a)}$ of shape $\vec{\lambda}$ and type $(c,a)$ is created by 
placing a rim hook $\zeta_1$ in $\vec{\lambda}$ of length $c_1$, placing a rim hook $\zeta_2$ in 
$\vec{\lambda} \setminus \zeta_1$ of length $c_2$, placing a rim hook $\zeta_3$ in 
$\vec{\lambda} \setminus \zeta_1 \setminus \zeta_2$ of length $c_3$, and so on. We then assign each rim hook a weight $w$ given by 
\begin{align*}
w(\zeta_i)&\coloneq \Psi^{j}(a_i) \sign(\zeta_i),
\end{align*}
where $\zeta_i$ was inserted in partition $\lambda^j$, and set 
\begin{align*}
w(T)&\coloneq \prod_{i=1}^{\ell(c)} w(\zeta_i).
\end{align*}

As an example for when $k=2$ and $n=8$, Figure~\ref{fig:rimhooktableaux} lists all rim hook tableaux
of shape $((3,1),(2,2))$ and type $(1,3,1,3),(1,-1,-1,1)$, where the rim hooks are numbered to describe the order in which they were placed.
The figure is displayed so that the tableaux of weight $+1$ are in the first row, and the tableaux with weight $-1$ are in the second row. 
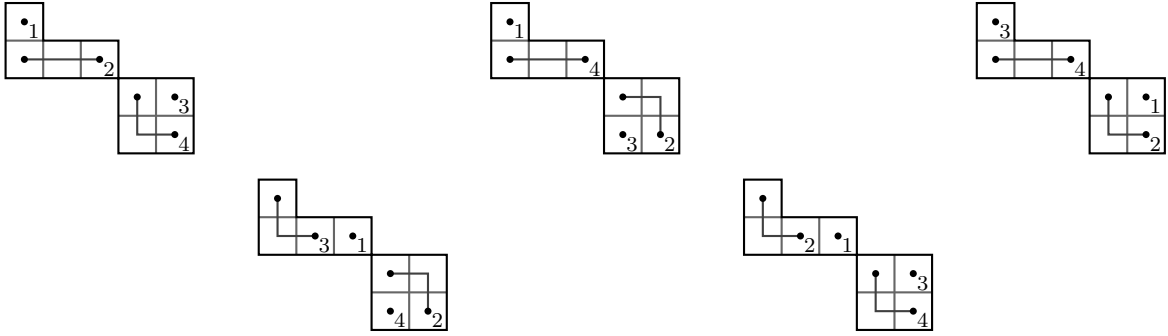
\begin{figure}
 \begin{align*}
 \begin{tikzpicture}[scale = 1/2, baseline = 10ex]
\draw [color = black!60, thick] (4,0) -- (4,2);
\draw [color = black!60, thick] (3,1) -- (5,1);
\draw [color = black!60, thick] (2,2) -- (2,3);
\draw [color = black!60, thick] (1,2) -- (1,3);
\draw [color = black!60, thick] (0,3) -- (1,3);
\draw[color=black!75, thick] (.5, 2.5) -- (2.5, 2.5);
\fill[color = black, thick] (.5,3.5) circle (.5ex);
\fill[color = black, thick] (.5,2.5) circle (.5ex);
\fill[color = black, thick] (2.5,2.5) circle (.5ex);
\draw[color=black!75, thick] (3.5, 1.5) -- (3.5, .5) -- (4.5,.5);
\fill[color = black, thick] (3.5,1.5) circle (.5ex);
\fill[color = black, thick] (4.5,.5) circle (.5ex);
\fill[color = black, thick] (4.5,1.5) circle (.5ex);
\node at (.75,3.25) {\scriptsize{$1$}};
\node at (2.75,2.25) {\scriptsize{$2$}};
\node at (4.75,1.25) {\scriptsize{$3$}};
\node at (4.75,.25) {\scriptsize{$4$}};
\draw [black, thick] (0,2) -- (5,2)  -- (5,0) -- (3,0) -- (3, 3) -- (1,3) -- (1,4) --(0,4) -- cycle;
\end{tikzpicture}
&& \begin{tikzpicture}[scale = 1/2, baseline = 10ex]
\draw [color = black!60, thick] (4,0) -- (4,2);
\draw [color = black!60, thick] (3,1) -- (5,1);
\draw [color = black!60, thick] (2,2) -- (2,3);
\draw [color = black!60, thick] (1,2) -- (1,3);
\draw [color = black!60, thick] (0,3) -- (1,3);
\draw[color=black!75, thick] (.5, 2.5) -- (2.5, 2.5);
\fill[color = black, thick] (.5,3.5) circle (.5ex);
\fill[color = black, thick] (.5,2.5) circle (.5ex);
\fill[color = black, thick] (2.5,2.5) circle (.5ex);
\draw[color=black!75, thick] (3.5, 1.5) -- (4.5, 1.5) -- (4.5,.5);
\fill[color = black, thick] (3.5,1.5) circle (.5ex);
\fill[color = black, thick] (4.5,.5) circle (.5ex);
\fill[color = black, thick] (3.5,.5) circle (.5ex);
\node at (.75,3.25) {\scriptsize{$1$}};
\node at (2.75,2.25) {\scriptsize{$4$}};
\node at (3.75,.25) {\scriptsize{$3$}};
\node at (4.75,.25) {\scriptsize{$2$}};
\draw [black, thick] (0,2) -- (5,2)  -- (5,0) -- (3,0) -- (3, 3) -- (1,3) -- (1,4) --(0,4) -- cycle;
\end{tikzpicture}
&&  \begin{tikzpicture}[scale = 1/2, baseline = 10ex]
\draw [color = black!60, thick] (4,0) -- (4,2);
\draw [color = black!60, thick] (3,1) -- (5,1);
\draw [color = black!60, thick] (2,2) -- (2,3);
\draw [color = black!60, thick] (1,2) -- (1,3);
\draw [color = black!60, thick] (0,3) -- (1,3);
\draw[color=black!75, thick] (.5, 2.5) -- (2.5, 2.5);
\fill[color = black, thick] (.5,3.5) circle (.5ex);
\fill[color = black, thick] (.5,2.5) circle (.5ex);
\fill[color = black, thick] (2.5,2.5) circle (.5ex);
\draw[color=black!75, thick] (3.5, 1.5) -- (3.5, .5) -- (4.5,.5);
\fill[color = black, thick] (3.5,1.5) circle (.5ex);
\fill[color = black, thick] (4.5,.5) circle (.5ex);
\fill[color = black, thick] (4.5,1.5) circle (.5ex);
\node at (.75,3.25) {\scriptsize{$3$}};
\node at (2.75,2.25) {\scriptsize{$4$}};
\node at (4.75,1.25) {\scriptsize{$1$}};
\node at (4.75,.25) {\scriptsize{$2$}};
\draw [black, thick] (0,2) -- (5,2)  -- (5,0) -- (3,0) -- (3, 3) -- (1,3) -- (1,4) --(0,4) -- cycle;
\end{tikzpicture}
\\
& \rule{2eM}{0eM}
\begin{tikzpicture}[scale = 1/2, baseline = 10ex]
\draw [color = black!60, thick] (4,0) -- (4,2);
\draw [color = black!60, thick] (3,1) -- (5,1);
\draw [color = black!60, thick] (2,2) -- (2,3);
\draw [color = black!60, thick] (1,2) -- (1,3);
\draw [color = black!60, thick] (0,3) -- (1,3);
\draw[color=black!75, thick] (.5, 3.5) -- (.5, 2.5) -- (1.5,2.5);
\fill[color = black, thick] (.5,3.5) circle (.5ex);
\fill[color = black, thick] (2.5,2.5) circle (.5ex);
\fill[color = black, thick] (1.5,2.5) circle (.5ex);
\draw[color=black!75, thick] (3.5, 1.5) -- (4.5, 1.5) -- (4.5,.5);
\fill[color = black, thick] (3.5,1.5) circle (.5ex);
\fill[color = black, thick] (4.5,.5) circle (.5ex);
\fill[color = black, thick] (3.5,.5) circle (.5ex);
\node at (1.75,2.25) {\scriptsize{$3$}};
\node at (2.75,2.25) {\scriptsize{$1$}};
\node at (4.75,.25) {\scriptsize{$2$}};
\node at (3.75,.25) {\scriptsize{$4$}};
\draw [black, thick] (0,2) -- (5,2)  -- (5,0) -- (3,0) -- (3, 3) -- (1,3) -- (1,4) --(0,4) -- cycle;
\end{tikzpicture}
&& \rule{2eM}{0eM}
\begin{tikzpicture}[scale = 1/2, baseline = 10ex]
\draw [color = black!60, thick] (4,0) -- (4,2);
\draw [color = black!60, thick] (3,1) -- (5,1);
\draw [color = black!60, thick] (2,2) -- (2,3);
\draw [color = black!60, thick] (1,2) -- (1,3);
\draw [color = black!60, thick] (0,3) -- (1,3);
\draw[color=black!75, thick] (.5, 3.5) -- (.5, 2.5) -- (1.5,2.5);
\fill[color = black, thick] (.5,3.5) circle (.5ex);
\fill[color = black, thick] (2.5,2.5) circle (.5ex);
\fill[color = black, thick] (1.5,2.5) circle (.5ex);
\draw[color=black!75, thick] (3.5, 1.5) -- (3.5, .5) -- (4.5,.5);
\fill[color = black, thick] (3.5,1.5) circle (.5ex);
\fill[color = black, thick] (4.5,.5) circle (.5ex);
\fill[color = black, thick] (4.5,1.5) circle (.5ex);
\node at (1.75,2.25) {\scriptsize{$2$}};
\node at (2.75,2.25) {\scriptsize{$1$}};Symmetric functions and the Frobenius map
\node at (4.75,1.25) {\scriptsize{$3$}};
\node at (4.75,.25) {\scriptsize{$4$}};
\draw [black, thick] (0,2) -- (5,2)  -- (5,0) -- (3,0) -- (3, 3) -- (1,3) -- (1,4) --(0,4) -- cycle;
\end{tikzpicture}
\end{align*}
\caption{All rim hook tableaux of shape $((3,1),(2,2))$ and type $(1,3,1,3),(1,-1,-1,1)$. }
\label{fig:rimhooktableaux}
\end{figure}

\begin{proposition}[The wreath Murnaghan-Nakayama rule \cite{Stembridge,MRW}]
Let $c(\sigma)$ be the cycle structure of $\sigma \in Z_k \wr S_n$ written in some fixed order, and $a(\sigma)$ be the C-types of the respective
cycles. Then, the irreducible character indexed by $\vec{\lambda}$ is given by
\begin{equation*}
\chi^{\vec{\lambda}}( \sigma ) = \sum_{T \in \RHT_{\vec{\lambda},(c(\sigma),a(\sigma))}} w(T).
\end{equation*}
\end{proposition}
This means that, from the previous example, 
\[\text{ if }
\sigma \in C_{((3,1),(2,2))},
\text{
then } \chi^{((3,1),(2,2))}(\sigma) = 1.\]

\section{Color rules and the main result}\label{section:color-rules}
We assume that we are working with some ring, say $R=\mathbb{C}((t_1,t_2,\dots))$, with indeterminates $t_1,t_2,\dots$; and we are interested in computing a class function $\chi: Z_k \wr S_n \rightarrow R$. It is also important to mention that we will use brackets to denote multisets, so that $[1^2,2] = [1,1,2]= [1,2,1]$ and $[1,\emptyset]\neq[1]$.

\subsection{Definitions}

\begin{definition}[Color rules]
A {\emph{color rule}} is a multiset of colors $F= [f_1^{m_1},f_2^{m_2},\dots]$.
A coloring of an element $\sigma \in Z_k \wr S_n$ under the color rule is a pair $(f,\sigma)$ where $f = (f_{r_1},\dots,f_{r_n})$ gives a multiset $[f_{r_1},\dots,f_{r_n}]$ contained (with multiplicities) in $F$ and satisfies the
following condition:

\begin{center} 
If $i$ and $j$ are indices in the same cycle of $\sigma$, then $f_{r_i} = f_{r_j}$.
\end{center}

The set of all colorings by the coloring rule $F$ will be denoted by $F(Z_k \wr S_n)$. \\

Given two color rules $F$ and $G$,
the set of colorings $F\times G(Z_k \wr S_n)$ is defined as the set of pairs
$(h,\sigma)$ where $h= ((f_{r_1},g_{s_1}),\dots, (f_{r_n},g_{s_n}))$ is composed of two colorings $f =(f_{r_1},\dots,f_{r_n})$ and $g = (g_{s_1},\dots,g_{s_n})$. That is, $((f,\sigma),(g,\sigma)) \in F(Z_k \wr S_n) \times G(Z_k \wr S_n)$. In particular, this means that if $i$ and $j$ are in the same cycle of $\sigma$, then $f_{r_i} = f_{r_j}$ and $g_{s_i} = g_{s_j}$.
\end{definition}

\begin{definition}\label{definition:color_rule}
A class function $\chi: Z_k \wr S_n \rightarrow R$ is given by the coloring rule $F$ if there is
\begin{enumerate}
     \item a ``value'' function $p_F:F \rightarrow \mathbb{Z}$ and 
     \item a ``weight'' function $\rho_F:F \rightarrow R$
\end{enumerate}
so that
$\chi(\sigma)$ can be computed as
\[
\chi(\sigma) = \sum_{(g,\sigma) \in F(Z_k \wr S_n)}  \weight(g,\sigma),
\]
where
\begin{align*}
\weight(g,\sigma) \coloneq \prod_{i=1}^n \weight(g_i,u(\sigma,i))&& \text{ and } &&
    \weight(f_i,u_j) \coloneq u_j^{p(f_i)} \rho(f_i).
\end{align*}
When the color rule, weight, and values are clear, we will write $p = p_F$ and $\rho = \rho_F$. 
\end{definition}
The following proposition follows from the definition and states that tensor products of representations whose characters are given by color rules have characters which are also given by a color rule.
\begin{proposition} \label{proposition:products_of_colors}
Suppose we have two class functions,
$\chi^M$ and $\chi^{M'}$, given by two color rules $F$ and $G$, respectively, with values $p_F,p_G$ and weights $\rho_F,\rho_G$. Then the class function $\chi^M \chi^{M'} $ is given by the color rule
$F \times G$ with value $p(f,g) = p_F(f)+ p_G(g)$ and weight $\rho(f,g)= \rho_F(f)\rho_G(g).$
In particular, if $\chi^M$ and $\chi^{M'}$ are graded characters for some modules $M$ and $M'$, then the graded character of $M \otimes M'$ is given by the color rule $F \times G$.
\end{proposition}

\subsection{Examples}

\begin{example} \label{example:character_of_defining}
    The most basic case is already remarkably interesting, as we will see in Subsection~\ref{section:wreath_defining_representation}. Let $F = [1,\emptyset^{n-1}]$, meaning we must color a $1$-cycle with the color $1$. Let $p(1) = \rho(1)=\rho(\emptyset) = 1$ and $p(\emptyset) = 0$. The reason this is interesting is that then
    \[
    \sum_{(g,\sigma) \in F(Z_k \wr S_n)} \weight(g,\sigma) = \sum_{\sigma_i = i} u(\sigma,i) ~~= \tr( \Pi_{\defining}(\sigma))
    \]
    gives the trace of the matrix which defines the element $\sigma$. It is therefore the character of the defining representation. \\
    
    Note that taking the character $\tr^m \Pi_{\defining}$ corresponds to the $m$-fold tensor product of the defining representation. The character $\tr^m \Pi_{\defining}(\sigma)$ can be computed by coloring the cycles of $\sigma$ in the following way: over each cycle of length $1$ place a subset of $\{1,\dots,m\}$ so that no number is ever repeated. The value is then given by
    \[
    \tr^m \Pi_{\defining}(\sigma) = \sum_{S } ~~\prod_{ \sigma_i = i }
    u(\sigma,i)^{|S_i|},
    \]
    where the sum ranges over all $S = (S_1,\dots,S_n)$ forming a set partition $\cup S_i = \{1,\dots,m\}$ ($S_i \cap S_j = \emptyset$ for $i \neq j$), and $S_i = \emptyset$ if $i$ is not in a $1$-cycle ($\sigma_i \neq i$).

    For instance, for $n=8$, $k = 3$, and $m=3$, the following is an example of a coloring:
    \begin{center}
    \begin{tikzpicture}[scale = 3/4]
    \node at (.75,0) {$(1)$};
    \node at (2,0) {$(u_2 2)$};
    \node at (4,0) {$(3 ~ u_1 6~ 4)$};
    \node at (6,0) {$(u_1 5)$};
    \node at (8,0) {$(u_2 8 ~ u_17)$};
    \node at (.75,.6) {$\{2\}$};
    \node at (2,.6) {$\emptyset$};
    \node at (4,.6) {$\emptyset$};
    \node at (4.5,.6) {$\emptyset$};
    \node at (3.5,.6) {$\emptyset$};
    \node at (6,.6) {$\{1,3\}$};
    \node at (8.5,.6) {$\emptyset$};
    \node at (7.5,.6) {$\emptyset$};
    \end{tikzpicture}.
    \end{center}
    We see that the weight of this coloring is $u_0^{|\{2\}|}  \times u_1^{|\{1,3\}|}= u_2$.
\end{example}

\begin{example}
    One of the main examples we will see is given by the colors
    \[
    \mathbb{N}_d \coloneq \left[ 0^n, 1^n, \ldots, d^n \right].
    \]
    Define
    \begin{align*}
    p(i) = i && \text{ and } && \rho(i) = q^i.
    \end{align*}
    We will see that the sum
    \[
   \sum_{d \geq 0} t^d \sum_{(g,\sigma) \in \mathbb{N}_d(Z_k \wr S_n) } \weight(g,\sigma)
    \]
    is the character of a certain semigroup algebra, $K[(\Delta_1)^{\times n}]$ (see Theorem~\ref{theorem:equivariant-euler-mahonian}).
\end{example}

\subsection{The main result}

We will now state the main result regarding the decomposition of class functions given by color rules into irreducible characters, and we defer its proof to Section~\ref{section:main_proof}.
Even though the proof is noteworthy on its own, it will be more useful to cover how this result can be applied, especially to the objects we are most interested in.

\begin{definition} \label{definition:SSTYk}
Let $F = [f_1^{m_1},f_2^{m_2},\dots]$ be a color rule with value function $p$ and weight function $\rho$. Let $F_r= [ f_i^{m_i} \in F : p(f_i) = r \mod k]$  be the submultiset of $r$-valued colors. 
For $\vec{\gamma} \vdash_k n$,
let 
\[
\SSYT_k(\vec{\gamma},F)\coloneq \SSYT(\gamma^0,F_0)\times \cdots \times 
\SSYT(\gamma^{k-1},F_{k-1})
\]
denote the set of all semistandard tableaux of shape $\vec{\gamma}$ whose entries in $\gamma^r$ form a submultiset of $F_r$. Furthermore, define for $T\in\SSYT_k(\vec{\gamma},F)$ the weight 
\[
\rho(T) = \prod_{f_i} \rho(f_i)^{a_i}
\]
where $a_i$ is the number of times $f_i$ appears in $T$.
\end{definition}

Then we have the following decomposition:
\begin{theorem} \label{theorem:mainresult} Let $\chi$ be given by a color rule $F$ with value $p$ and weight $\rho$. 
For any $\vec{\gamma} \vdash_k n$,
\[
\langle \chi, \chi^{\vec{\gamma}} \rangle = \sum_{T \in \SSYT_k(\vec{\gamma},F)} \rho(T).
\]
\end{theorem}

\section{Some examples}\label{section:some_examples}
We now list some applications.
It is important to remark that from Proposition~\ref{proposition:products_of_colors}, we can immediately take any arbitrary tensor products of the representations which are subsequently listed here and immediately give a combinatorial description of their decomposition into irreducible representations.

\subsection{The symmetric group}
\label{subsection:symmetric_group}
When $k=1$, we are looking at characters of the symmetric group. Since, $\vec{\gamma}  = \lambda$ will consist of only one partition, we will write $\SSYT(\lambda,F)$ instead of $\SSYT_1(\vec{\gamma},F)$. 

\subsubsection{Symmetric functions and the Frobenius map} \label{section:symmetric_functions}
It will be useful to recall
the following connection between class functions of the symmetric group and symmetric functions. 
For us, class functions are maps $\chi:S_n \rightarrow R = \mathbb{C}((t_1,t_2,\dots))$, which are constant on conjugacy classes. And our symmetric functions will have coefficients in $R$.

Let $C_\lambda$ be the conjugacy class of $S_n$ with permutations of cycle type $\lambda$, which is also viewed as the class function which indicates if a permutation is in $C_\lambda$. The Frobenius map, from class functions to symmetric functions, is defined by
\[
\Frob(C_\lambda) = p_\lambda/z_\lambda,
\]
where $z_\lambda = n!/|C_\lambda|$ and $p_\lambda = p_{\lambda_1} \cdots p_{\lambda_{\ell(\lambda)}}$ is the power sum symmetric function generated
by $p_r = x_1^r + x_2^r+ \cdots$. 
Then the Frobenius map sends a class function $\chi = \sum_\lambda \chi(\lambda) C_\lambda$ to
\[
\Frob(\chi) = \sum_{\lambda \vdash n} \chi(\lambda) p_\lambda/z_\lambda,
\]
where $\chi(\lambda)$ is the value of $\chi$ on $C_\lambda$.
A well-known property of the Frobenius map is that the irreducible character $\chi^\lambda$ is sent to 
\[
\Frob(\chi^\lambda) = s_\lambda \coloneq \sum_{T \in \SSYT(\lambda, [1^n,2^n,\dots])} \rho(T),
\]
where the color rule from $[1^n,2^n,\dots]$ has weight $\rho(i) =x_i$. This is, in fact, also a consequence of Theorem~\ref{theorem:mainresult}.
The (complete) homogeneous symmetric function 
\[
h_n = \sum_{a_1 \leq \cdots \leq a_n} x_{a_1} \cdots x_{a_n}
\]
then corresponds to the trivial representation. Theorem~\ref{theorem:mainresult} also would imply that if a module $M$ is isomorphic to $1 \uparrow_{S_{\mu_1} \times \cdots \times S_{\mu_\ell}}^{S_n}$, that is $M$ is generated by a single element whose stabilizer is isomorphic to $S_{\mu_1}\times \cdots \times S_{\mu_\ell}$, then
\[
\Frob\left(1 \uparrow_{S_{\mu_1} \times \cdots \times S_{\mu_\ell}}^{S_n}\right)  = h_\mu,
\]
which is further addressed in Section~\ref{section:young_tabloid}.
Before looking at our examples, we would like to warn readers that we will be making light use of plethystic substitution. The important part to remember is that if $y_1,y_2,\dots$ are positive monomials in $R$, then $F[y_1+y_2+\cdots]$ corresponds to variables substitution $F(y_1,y_2,\dots)$ in the symmetric function. More generally, when not all the terms are positive monomials, one must first expand $F = \sum_\lambda c_\lambda p_\lambda$ in terms of power sums, then set, for any expression $E$,
\[
p_r[E(t_1,t_2,\dots)] = E(t_1^r,t_2^r,\dots).
\]
When $X= x_1+x_2+\cdots$, we get $p_r [X] = x_1^r+x_2^r+\cdots = p_r$, so that for any symmetric function $F$, we have $F[X] = F.$

\subsubsection{The defining representation}

The character for the defining representation of $S_n$ is given by $\sigma \mapsto \fxd(\sigma)$, the number of fixed points of $\sigma$. 
Example~\ref{example:character_of_defining} describes the color rule:
$\fxd^m$ is computed by choosing a set partition $\{S_1,\dots, S_r\}$ of $\{1,\dots, m\}$, then placing each $S_i$ over a unique fixed point. This means $\SSYT(\lambda,[1,\emptyset^{n-1}]^{\times m} )$ can be interpreted as the set of semistandard tableaux whose entries (which are possibly empty) form a set partition of $\{1,\dots, m\}$.
Theorem~\ref{theorem:mainresult} specializes to the combinatorial proof of the following result given in \cite{MendesRomero}.
\begin{proposition}
For $\lambda \vdash n $, 
\[
\langle \chi^\lambda, \fxd^m \rangle = 
| \SSYT(\lambda, [1,\emptyset^{n-1}]^{\times m})|.
\]
\end{proposition}
For instance, suppose $n= 5,$  $m = 3$, and $\lambda = (3,2)$. Then $\SSYT(\lambda, [1,\emptyset^{4}]^{\times 3})$ would contain the following semistandard tableaux, where the order on colors is given by the minimal element of each set:
\begin{align*}
\scalebox{.8}{
\begin{tikzpicture}[scale=1]
\partitionfr{3,2};
\draw (.5,.5) node {{$\{1,2\} $}};
\draw (1.5,.5) node {{$\{3\} $}};
\end{tikzpicture}} &&
\scalebox{.8}{
\begin{tikzpicture}5
\partitionfr{3,2};
\draw (.5,.5) node {$\{1,3\}$ };
\draw (1.5,.5) node {$\{2\} $};
\end{tikzpicture} }&&
\scalebox{.8}{
\begin{tikzpicture}
\partitionfr{3,2};
\draw (.5,.5) node {$\{1\}$ };
\draw (1.5,.5) node {$\{2,3\} $};
\end{tikzpicture}} \\
\scalebox{.8}{
\begin{tikzpicture}
\partitionfr{3,2};
\draw (.5,.5) node {$\{1\} $};
\draw (1.5,.5) node {$\{2\} $};
\draw (2.5,.5) node {$\{3\} $};
\end{tikzpicture}} &&
\scalebox{.8}{
\begin{tikzpicture}
\partitionfr{3,2};
\draw (.5,.5) node {$\{1\}$ };
\draw (1.5,.5) node {$\{2\} $};
\draw (.5,1.5) node {$\{3\} $};
\end{tikzpicture}} &&
\scalebox{.8}{
\begin{tikzpicture}
\partitionfr{3,2};
\draw (.5,.5) node {$\{1\}$ };
\draw (1.5,.5) node {$\{3\} $};
\draw (.5,1.5) node {$\{2\} $};
\end{tikzpicture} }
\end{align*}
This means that
\[
\langle \chi^{(3,2)}, \fxd^3 \rangle = 6.
\]

\subsubsection{Actions on words}

Let $F = [1^{n},\dots, N^n]$. This means every cycle can get any color.
Given a word $a_1\cdots a_{n} \in \{1,\dots, N\}^n$ from $N$ letters, we have the action of $S_n$ on the positions, given by 
\[
\sigma(a_1\cdots a_n) = a_{\sigma_1} \cdots a_{\sigma_n}.
\]
Let $\chi$ be the graded character of this action, given by
\[
\chi(\sigma) = \sum_{\sigma(a_1\cdots a_n)  = a_1\cdots a_n} t_{a_1} \cdots t_{a_n}.
\]
Then $\chi$ is given by the color rule $F$ with $\rho(f_i) = t_{f_i}.$
As a consequence of Theorem~\ref{theorem:mainresult}, we have
\[
\langle  \chi^{\lambda}, \chi \rangle = \sum_{T \in \SSYT(\lambda, [1^n,\dots, N^n] ) } \rho(T) = s_\lambda[t_1+\cdots + t_N].
\] 
This gives
\[
s_{\lambda}[t_1+\dots+ t_N] = \left\langle s_{\lambda},  \Frob(\chi) \right \rangle,
\]
where $\Frob$ is the Frobenius characteristic map in Subsection~\ref{section:symmetric_functions}.
Therefore,
$$
 \Frob(\chi) = \sum_{\lambda \vdash n} s_{\lambda}[t_1+ \cdots + t_N] s_{\lambda}[X] = h_n\left[ (t_1+\cdots + t_N) X\right].
$$

When we take the $m^{\text{th}}$ Kronecker power, we use as colors $[1^n,\dots, N^n]^{\times m}$ giving
\[
\Frob(\chi^{ m}) = \sum_{\lambda \vdash n} s_\lambda[(t_1+\cdots +t_N)^m ]s_{\lambda}[X] = h_n\left[ (t_1+\cdots + t_N)^m X \right].
\]
Setting $t_i =1$ gives the multiplicities of irreducible representations in the $m$-fold tensor product of the representation arising from the action (on positions) of the symmetric group on words with $N$ letters.
\subsubsection{Young tabloids} \label{section:young_tabloid}
We take the action of $S_n$ on the set of row strict (increasing) tableaux $\RST(\mu)$ of shape $\mu$ and with entries $1,\dots, n$.  The action of $\sigma$ on $T \in \RST(\mu)$ is given by sending $i$ in $T$ to $\sigma_i$ , then rearranging the rows of the tableau in increasing order. We will denote the character of this action by $\psi_\mu$. The character is given by sending $\sigma \in S_n$ to
$$
\psi_\mu(\sigma) = |\{T \in \RST(\mu): \sigma T = T \}|.
$$
A row strict tableau $T$ is 
fixed by $\sigma$ if and only if the following holds:
\begin{center}
If $i$ and $j$ are in the same cycle of $\sigma$, then $i$ and $j$ are in the same row of $T$.
\end{center}
 For a partition $\mu$, let $[\mu] = [1^{\mu_1},\dots, n^{\mu_n}]$, with $p(f) = \rho(f) = 1$ trivial. This color rule will then give the character $\psi_\mu$. Our main result gives
\[
\langle \chi_\mu, \chi^\lambda \rangle = |\SYT(\lambda, [\mu])| = K_{\lambda,\mu},
\]
the Kostka numbers.
 Therefore, 
\[
\Frob(\psi_\mu) = \sum_{\lambda \vdash n} K_{\lambda,\mu} s_{\lambda}[X] = h_{\mu}.
\]

For any two partitions $\mu, \nu \vdash n$, let 
\begin{align*}
[\mu] \star [\nu] = \{ [(a_1,b_1),\dots, &(a_n,b_n)]: \\ &
[a_1,\dots, a_n] = [\mu] \text{ and } [b_1,\dots, b_n]= [\nu]\}
\end{align*}

Then 
\[
\langle \psi_\mu \cdot \psi_\nu~, \chi^{\lambda} \rangle = |{\SYT(\lambda, \mu \star \nu)}| ~ =  ~\langle h_\mu \ast h_{\nu} , s_{\lambda} \rangle
\]
where $\SYT(\lambda, [\mu] \star [\nu])$ is the set of semistandard tableaux whose entries, with multiplicities, are given by
 $[(a_1,b_1),\dots, (a_n,b_n)]$ for some $[(a_1,b_1),\dots, (a_n,b_n)] \in [\mu] \star [\nu]$. It is easier to see what we mean by example. So let $\mu = (3,2)$ and $\nu = (2,2,1)$. Then $[1^{\mu_1},\dots,n^{\mu_n}] = [1^3, 2^2] = [1,1,1,2,2]$, and likewise $[1^{\nu_1},\dots n^{\nu_n}] = [1,1,2,2,3]$. The $\star$ product produces the set $[\mu] \star [\nu]$
 \begin{align*}
  = \{  [(1,1),(1,1),(1,2),(2,2),(2,3)]  &, & [(1,1),(1,2),(1,2),(2,1),(2,3)], \\
 [(1,1),(1,1),(1,3),(2,2),(2,2)]  &, & [(1,1),(1,2),(1,3),(2,1),(2,2)], \\
  &&  [(1,2),(1,2),(1,3),(2,1),(2,1)] \}.
 \end{align*}

So to compute $ \langle h_{3,2} \ast h_{2,2,1} , s_{2,1,1,1} \rangle$ we must compute the number of semistandard Young tableaux of shape $(2,1,1,1)$ in each of the above multisets. It is not difficult to check that the corresponding quantities are given respectively by $1,1,0,4,$ and $0$,
the tableaux being listed in Figure~\ref{figure:kronecker_homogeneous_tableaux}.
 
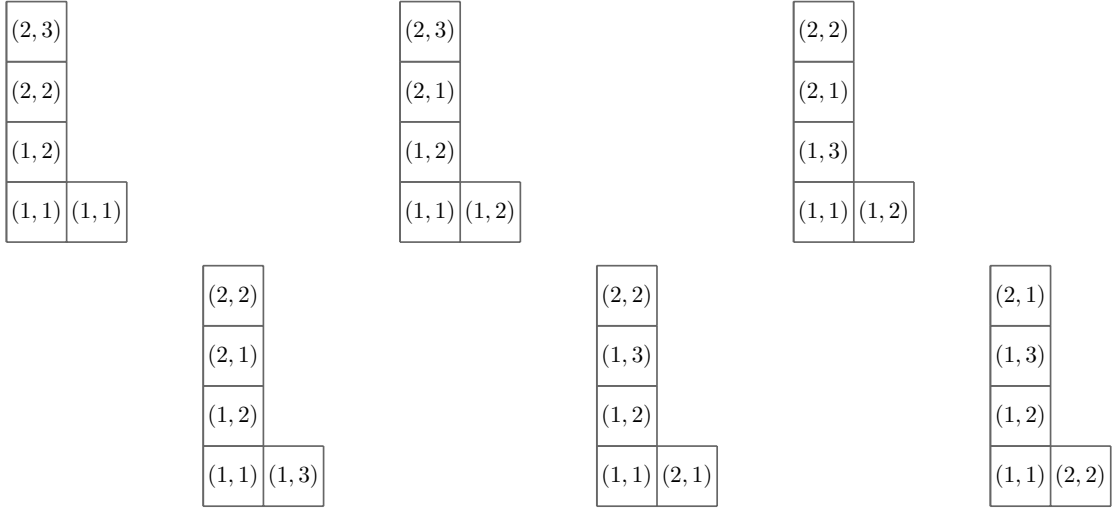
\begin{figure}
\begin{align*}
\scalebox{.8}{
\begin{tikzpicture}
\partitionfr{2,1,1,1};
\draw (.5,.5) node {$(1,1) $};
\draw (1.5,.5) node {$(1,1) $};
\draw (.5,1.5) node {$(1,2) $};
\draw (.5,2.5) node {$(2,2) $};
\draw (.5,3.5) node {$(2,3) $};
\end{tikzpicture} }&& &&
\scalebox{.8}{
\begin{tikzpicture}
\partitionfr{2,1,1,1};
\draw (.5,.5) node {$(1,1) $};
\draw (1.5,.5) node {$(1,2) $};
\draw (.5,1.5) node {$(1,2) $};
\draw (.5,2.5) node {$(2,1) $};
\draw (.5,3.5) node {$(2,3) $};
\end{tikzpicture} }&& &&
\scalebox{.8}{
\begin{tikzpicture}
\partitionfr{2,1,1,1};
\draw (.5,.5) node {$(1,1) $};
\draw (1.5,.5) node {$(1,2) $};
\draw (.5,1.5) node {$(1,3) $};
\draw (.5,2.5) node {$(2,1) $};
\draw (.5,3.5) node {$(2,2) $};
\end{tikzpicture} }\\
&&
\scalebox{.8}{
\begin{tikzpicture}
\partitionfr{2,1,1,1};
\draw (.5,.5) node {$(1,1) $};
\draw (1.5,.5) node {$(1,3) $};
\draw (.5,1.5) node {$(1,2) $};
\draw (.5,2.5) node {$(2,1) $};
\draw (.5,3.5) node {$(2,2) $};
\end{tikzpicture} }&& &&
\scalebox{.8}{
\begin{tikzpicture}
\partitionfr{2,1,1,1};
\draw (.5,.5) node {$(1,1) $};
\draw (1.5,.5) node {$(2,1) $};
\draw (.5,1.5) node {$(1,2) $};
\draw (.5,2.5) node {$(1,3) $};
\draw (.5,3.5) node {$(2,2) $};
\end{tikzpicture} }&& &&
\scalebox{.8}{
\begin{tikzpicture}
\partitionfr{2,1,1,1};
\draw (.5,.5) node {$(1,1) $};
\draw (1.5,.5) node {$(2,2) $};
\draw (.5,1.5) node {$(1,2) $};
\draw (.5,2.5) node {$(1,3) $};
\draw (.5,3.5) node {$(2,1) $};
\end{tikzpicture} }
\end{align*}
\caption{A list of all tableaux of shape $(2,1,1,1)$ with labels coming from $[(3,2)] \star[(2,2,1)].$}
\label{figure:kronecker_homogeneous_tableaux}
\end{figure}

Therefore,
\[
\langle h_{3,2} \ast h_{2,2,1} , s_{2,1,1,1} \rangle = 6.
\]

For $\mu^1,\dots, \mu^s \vdash n$ we can similarly define the product $[{\mu^1}]\star [{\mu^2} ] \star \cdots \star [{\mu^{s}}]$. We get 
\begin{corollary}\label{cor:hstars} For $\mu^1,\dots, \mu^s \vdash n$, 
\[
\langle h_{\mu^1} \ast \cdots \ast h_{\mu^s}, s_{\lambda} \rangle = | SYT(\lambda, [\mu^1] \star \cdots \star [\mu^s])|.
\]
\end{corollary}

\noindent We should note that Corollary~\ref{cor:hstars}, when $s=2$, can also be seen by K Rule III in \cite{GarsiaRemmel}.

Another way to see why the action on $\RST(\mu)$ has character equal to that of $\psi_\mu$, we note that every element $(f_{r_1},\dots, f_{r_n})$ from $[\mu]$ which is fixed by the action of $\sigma$ represents a tableaux in $\RST(\mu)$ which is fixed by $\sigma$: place the label $i$ in row $f_{r_i}$ of $\mu$.

The action on $\RST(\mu)$ is transitive, and the stabilizer of any element is a row subgroup isomorphic to 
$
S_{\mu_1} \times S_{\mu_2} \times \cdots \times S_{\mu_{\ell(\mu)}},
$
which is sufficient to say the Frobenius image of this character is $h_{\mu_1} h_{\mu_2} \cdots h_{\mu_{\ell(\mu)}}$. On the other hand, Theorem~\ref{theorem:mainresult} also implies the following well-known fact:
\begin{proposition} Let $M$ be an $S_n$-module with character $\chi^M$. Suppose $S_n$ acts transitively on $M$, so that $M = \mathbb{C}[S_n] w$ is the orbit of a single element $w$.
Suppose the stabilizer of $w$ is isomorphic to $S_{\mu_1} \times \cdots S_{\mu_{\ell(\mu)}}$. Then
$$
\Frob(\chi^M) = h_{\mu}.
$$
\end{proposition}

We next define a coloring rule which is similar to this example. We will color our points by $[\mu]$, but we instead allow each point to get a multiset of colors.

\subsubsection{Restrictions from \texorpdfstring{$GL_n$}{GLn}}

Let $F= [ [1^{d_1},\dots, m^{d_m}]^{ n}: d_i \geq 0 ] $ be the set of colors consisting of multisets $[1^{d_1},\dots, m^{d_n}]$. For a coloring $f=(f_1,\dots,f_n)$ (so each $f_i$ is a multiset), let 
$$
\rho(f) = t_1^{a_1} \cdots t_{m}^{a_m}
$$
if, for each $i$, the number of occurrences of $i$ in all of $(f_1,\dots,f_n)$ is $a_i$.
\\

Take the polynomial ring $A = \mathbb{Q}[x_1,\dots, x_n]$ with the action of $S_n$ which sends $x_i$ under $\sigma$ to $x_{\sigma_i}$. Then the degree $d$ component of $A$, say $A_d$, is invariant under the action and produces a character $\psi_d$. Using the set of monomials $x_1^{r_1} \cdots x_n^{r_n}$ with $r_1+\cdots + r_n = d$ as a basis for $R_d$ gives a permutation action. To calculate
$$
\psi_d(\sigma) = |\{ x_1^{r_1} \cdots x_n^{r_n} :   x_1^{r_1} \cdots x_n^{r_n}= x_{\sigma_1}^{r_1} \cdots x_{\sigma_n}^{r_n} \text{ and } r_1+ \cdots + r_n = d\}|,
$$
we first note that $ x_1^{r_1} \cdots x_n^{r_n}$ is fixed by $\sigma$ provided the following holds:
$$
\text{ if $i$ and $j$ are in the same cycle of $\sigma$, then $r_i = r_j$. }
$$
We therefore must color the cycles of $\sigma$ by distributing the multiset $[1^d]$ over the cycles. The number of $1$'s above $i$ represents the power of $x_i$ in a monomial. For example, the following colored permutation would produce the monomial $ x_1^3 x_2 x_4^3 x_5 x_6^3 x_7^2 x_8^3 x_9  \in A_{17} $:
\begin{align*} 
\begin{tikzpicture}
\draw (1,0) node { (1};
\draw (2,0) node {4 };
\draw (3,0) node { 8};
\draw (4,0) node {6) };
\draw (5,0) node {(2 };
\draw (6,0) node {9 };
\draw (7,0) node { 5)};
\draw (8,0) node {(3) };
\draw (9,0) node {(7) };
\draw (1,.5) node {$[1^3]$};
\draw (2,.5) node {$[1^3]$ };
\draw (3,.5) node { $[1^3]$};
\draw (4,.5) node {$[1^3]$ };
\draw (5,.5) node {$[1] $};
\draw (6,.5) node {$[1]$ };
\draw (7,.5) node {$ [1] $};
\draw (8,.5) node {  };
\draw (9,.5) node {$[1^2 ]$ };
\end{tikzpicture}
\end{align*}
Since each element of a cycle has the same number of elements above, the produced monomial is fixed by $\sigma$. 
\\

The action of $GL_n$ on $A_1$ gives an action on $\Symm^d(A_1)= A_d$ whose character we denote by $\phi_d$. For $g \in GL_n$ with eigenvalues $\theta_1,\dots, \theta_n$, the character may be calculated by using the homogeneous symmetric function: 
$$
\phi_d(g) = h_d(\theta_1,\dots, \theta_n).
$$
In particular, when $g$ is a permutation matrix $\Pi(\sigma)$, where $\sigma$ has cycle structure $\mu$, then the eigenvalues of $\Pi(\sigma)$ with multiplicities are given by
$$
\zeta(\sigma)  = \left[1, \zeta_{\mu_1}, \dots, \zeta_{\mu_1}^{\mu_1-1}, 1, \zeta_{\mu_2}, \dots, \zeta_{\mu_2}^{\mu_2-1}, \dots,  1, \zeta_{\mu_{\ell(\mu)}}, \dots, \zeta_{\mu_{\ell(\mu)}}^{\mu_{\ell(\mu)}-1} \right] ,
$$ 
where $\zeta_r$ is an $r^{\text{th}}$ primitive root of unity. 
We get that 
$$
\psi_d(\sigma) = h_d(\zeta(\sigma)).
$$

It follows from our main result that
\begin{align*}
\langle \phi_{d_1} \cdots \phi_{d_m} \downarrow^{GL_n}_{S_n} , \chi^{\lambda} \rangle  & = \langle h_{d_1} \cdots h_{ d_m}~ \left[\zeta (\cdot)\right] , \chi^{\lambda} \rangle \\
& = \langle \psi_{d_1} \cdots \psi_{d_m} , \chi^{\lambda} \rangle \\
  &=\sum_{\cup M_i = [1^{d_1},\dots, m^{d_m}]} \sum_{T \in \SSYT(\lambda, [M_1,\dots,M_n])} \rho(T)  ~~
  \\
  &= s_\lambda\left[\frac{1}{(1-t_1)(1-t_2)\cdots (1-t_m)}\right]\Big|_{t_1^{d_1} \cdots t_m^{d_m}},
\end{align*}
where the sum runs over all semistandard with entries that are multisets $M_i$ whose union is $[1^{d_1}, 2^{d_2}, \cdots, m^{d_m}]$. 
The last equality interprets each $M_i = [1^{a_1},\dots, m^{a_m}]$ as coming from the series $1/((1-t_1)\cdots (1-t_m))$, and since the union equals of the $M_i$ equals $[1^{d_1}, 2^{d_2}, \cdots, m^{d_m}]$, we must take the coefficient of $t_1^{d_1}\cdots t_m^{d_m}$.

This decomposes the action of $S_n$ on the polynomial ring $
\mathbb{Q}[x_{i,j}:i=1,\dots, m \text{ and } j=1,\dots, n],
$
given by the Frobenius characteristic
\[
\sum_{\lambda \vdash n} s_\lambda s_\lambda\left[ \frac{1}{(1-t_1)\cdots (1-t_m)}\right] = h_n \left[ \frac{X}{(1-t_1)\cdots (1-t_m)}\right].
\]

This can be found in the work of Orellana and Zabrocki \cite{OrellanaZabrocki} and \cite{Romero}.

\subsection{Examples on \texorpdfstring{$Z_k \wr S_n$}{ZkSn}}
One can generalize many examples coming from $S_n$ to the wreath product case. However, we will focus on providing three main examples. The first is on tensor products of the defining representation; the second is on a combinatorial proof of the Murnaghan-Nakayama rule; and lastly, we will look at our main application to the affine semigroup algebras appearing in the subsequent section.

\subsubsection{The defining representation of \texorpdfstring{$Z_k \wr S_n$}{ZkSn}} \label{section:wreath_defining_representation}
Recall from Example~\ref{example:character_of_defining} that the character of the defining representation is given by the color rule $F = [1,\emptyset^{n-1}]$ with value and weight function given by $p(1) = 1$, $p(\emptyset) = 0$, and $\rho = 1$. 
We immediately see from Theorem~\ref{theorem:mainresult} that if $k >1$, then $\tr \Pi_{\defining} = \chi^{((n-1),(1),\emptyset,\dots, \emptyset)}$: there is only one semistandard tableau in the entries $[1,\emptyset^{n-1}]$ where $1$ is placed in $\gamma^1$ and all $\emptyset$'s are placed in $\gamma^0$, and this tableau has shape $((n-1),(1),\emptyset,\dots, \emptyset)$.

The product $[1,\emptyset^{n-1}]^{\times m}$ will then produce the colors for computing the trace of the $m$-fold tensor product of the defining representation, $\Pi_{\defining}^{\otimes m}$. For a given $\vec{\gamma} \vdash_k n$, the set
$\SSYT_k(\vec{\gamma},[1,\emptyset^{n-1}]^{\times m})$ can be interpreted as the set of all semsistandard tableaux of shape $\vec{\gamma}$ whose entries form a set partition of $\{1,\dots, m\}$, say $\{S_1,\dots, S_n \}$, such that $S_i$ appears in $\gamma^r$ if and only if $|S_i| = r  \mod k$.
For instance, Figure~\ref{fig:defining_tableaux_example} gives a tableau which contributes to the calculation of $\left \langle \chi^{((3,1),(2,2))}, \tr \Pi_{\defining}^{\otimes 10} \right \rangle.$

\begin{figure}[ht]
    \centering
\begin{align*}
    \begin{tikzpicture}[scale = 3/4, baseline = 10ex]
\draw [color = black!60, thick] (4,0) -- (4,2);
\draw [color = black!60, thick] (3,1) -- (5,1);
\draw [color = black!60, thick] (2,2) -- (2,3);
\draw [color = black!60, thick] (1,2) -- (1,3);
\draw [color = black!60, thick] (0,3) -- (1,3);
\node at (.5,2.5) {\tiny{$ 1,4 $}};
\node at (3.5,.5) {\tiny{$2$}};
\node at (1.5,2.5) {\tiny{$ 7,9 $}};
\node at (3.5,1.5) {\tiny{$3,5,6$}};
\node at (4.5,.5) {\tiny{$8$}};
\node at (4.5,1.5) {\tiny{$10$}};
\draw [black, thick] (0,2) -- (5,2)  -- (5,0) -- (3,0) -- (3, 3) -- (1,3) -- (1,4) --(0,4) -- cycle;
\end{tikzpicture}
\end{align*}
  \caption{An element of $\SSYT_2\left(((3,1),(2,2)), [1,\emptyset^{n-1}]^{\times 10}\right)$ interpreted using set partitions. Sets are drawn without brackets, and they are ordered by minimal element.}
    \label{fig:defining_tableaux_example}
\end{figure}

We immediately note something important: every cell of $\gamma^r$ for $r>0$ must contain at least $r$ elements; and at least $|\gamma^0|-\gamma^0_1$ cells in $\gamma^0$ must contain at least $k$ elements. Therefore, $\langle \chi^{\vec{\gamma}} , \tr \Pi_{\defining}^{\otimes m} \rangle = 0$ if $m < N(\vec{\gamma})$,
where 
\[
N(\vec{\gamma}) =k(|\gamma^0|-\gamma^0_1) + \sum_{i=1}^{k-1} i |\gamma^i| .
\]

Usually, determining if an irreducible representation appears in a tensor product is very difficult. Remarkably in this case, we are able to determine precisely when an irreducible representation appears!
\begin{corollary}
We have $\langle \tr \Pi_{\defining}^{\otimes m} , \chi^{\vec{\gamma}} \rangle \neq 0$
    if and only if $m = N(\vec{\gamma}) + kr$ for some $r\geq 0$.
\end{corollary}
\begin{proof}
    From the previous observations, the sum of minimal cardinalities for each cell in $\vec{\gamma}$ gives the smallest $m$ for which $\langle \chi^{\vec{\gamma}} , \tr \Pi_{\defining}^{\otimes m} \rangle \neq 0$:
$
m =N(\vec{\gamma}).
$

The cardinalities of the sets in each cell must always increase by at least $k$ from these minimal cardinalities. Therefore, if $\langle \chi^{\vec{\gamma}} , \tr \Pi_{\defining}^{\otimes m} \rangle \neq 0$, we must have $m = N(\vec{\gamma}) + kr$ for some $r\geq 0$. On the other hand, given such an $m$, we can always find a tableau in $\SSYT_k(\vec{\gamma},[1,\emptyset^{n-1}]^{\times m})$.
\end{proof}

\subsubsection[A combinatorial proof of the Murnaghan-Nakayama rule]{A combinatorial proof of the Murnaghan-Nakayama rule, multisymmetric functions, and the Frobenius map for wreath products}
\label{section:murnaghan-nakayama}
Suppose our color rule is given by $X^{\ast} = [x_{i,j}^n: i=0,\dots, k-1 \text{ and } j \geq 1]$. And suppose that the value of $x_{i,j}$ is given by $p(x_{i,j})  = i$, and that the weight is given by $\rho(x_{i,j} ) = x_{i,j}$. Then if we let $X^i = x_{i,1}+x_{i,2} + \cdots$,
\begin{align*}
\sum_{(\sigma,f) \in X^{\ast}(Z_k \wr S_n) } \weight(\sigma,f) & =   \prod_{r=0}^{k-1} 
\prod_{l=1}^{\ell(\lambda^r(\sigma))} \sum_{i=0}^{k-1} \sum_{j\geq 1} (u_r)^{p(x_{i,j})\lambda^r(\sigma)_l}  \rho(x_{i,j})^{\lambda^r(\sigma)_l} \\
& = 
\prod_{r=0}^{k-1} 
\prod_{l=1}^{\ell(\lambda^r(\sigma))} p_{\lambda^r(\sigma)_l} \left[ \sum_{i=0}^{k-1} { \Psi}^i(u_r) X^i \right].
\end{align*}

To simplify the notation, let us define for $\vec{\lambda}(\sigma) = \vec{\lambda}$ the multisymmetric function 
\[ P_{\vec{\lambda}(\sigma)}[X^0,\dots, X^{k-1}] = \prod_{r=0}^{k-1} p_{\lambda^r}\left[ \sum_{j=0}^{k-1} {\Psi}^j_r X^j \right]. \]

Then Theorem~\ref{theorem:mainresult} gives that 
\[ \frac{1}{n! k^n }
\sum_{\sigma \in Z_k \wr S_n} \overline{\chi}^{\vec{\gamma}}(\sigma) {{P}}_{\vec{\lambda}(\sigma)}[X^0,\dots, X^{k-1}] 
= \sum_{T \in \SSYT_k(\vec{\gamma}, X^{\ast})} \rho(T).
\]
But note that if we have the color $x_{i,j}$ in $\gamma^r$, then we are required to have $p(x_{i,j}) = i = r \mod k$. This means that each $\gamma^r$ can only be filled with elements of $X^r = x_{r,1} + x_{r,2} + \cdots$. In other words, we have found that
\begin{theorem}
For any $\vec{\gamma} \vdash_k n$, we have
\[ s_{\gamma^0}[X^0] s_{\gamma^1}[X^1] \cdots s_{\gamma^{k-1}}[X^{k-1}] = 
\sum_{\vec{\lambda} \vdash_k n} \overline{\chi}^{\vec{\gamma}}(\sigma) 
\frac{ {P}_{\vec{\lambda}}[X^0,\dots, X^{k-1}] }{ Z_{\vec{\lambda}} },
\]
where 
\[Z_{\vec{\lambda}} = \frac{n!k^n}{|C_{\vec{\lambda}}|} = 
 \prod_{i=0}^{k-1} \frac{z_{\lambda^i}}{k^{\ell(\lambda^i)}}.
\]

\end{theorem}

This is the symmetric function version of the Murnaghan-Nakayama rule, a variation of the identity found in \cite{MRW}. 
Here, we instead
consider the homomorphism from class functions of $\mathbb{C}[ Z_k \wr S_n ]$ to the $k$-fold tensor product of symmetric functions (or multisymmetric functions) given by sending a conjugacy class
\[
C_{\vec{\lambda}} \mapsto \frac{ \overline{P}_{\vec{\lambda}}[X^0,\dots, X^{k-1}]}{Z_{\vec{\lambda}}}.
\]
Then the irreducible character is sent to
\[
\chi^{\vec{\gamma}} \mapsto s_{\vec{\gamma}}[X] \coloneq \prod_{i=0}^{k-1} s_{\gamma^i}[X^i].
\]
Note here, that to simplify the notation, we will henceforth write
\begin{align*}
b_{\vec{\lambda}}[X] \coloneq \prod_{i=0}^{k-1} b_{\lambda^i}[X^i]
\end{align*}
when $b_\lambda$ is some classical symmetric function basis, such as the Schur functions.
\begin{proposition}
    The Frobenius map, $\Frob$, on characters of $Z_k \wr S_n$, given by
sending the conjugacy class $C_{\vec{\lambda}}$ to $\overline{P}_{\vec{\lambda}}[X]/Z_{\vec{\lambda}}$, satisfies
    \[
\Frob(\chi^{\vec{\gamma}}) = s_{\vec{\gamma}}[X].
    \]
\end{proposition}

\noindent For a given module $M$ with (graded) character $\chi^M$, we call $\Frob\left(\chi^M \right)$ the \emph{Frobenius characteristic} of $M$.

\subsubsection{A basis for all class functions} \label{subsection:basis}
Here, we show that every class function with values in $\mathbb{C}$ can be written as a sum of class functions given by color rules. 
For a given $\vec{\lambda} \vdash_k n$, let 
\[
F^{\vec{\lambda}} = \left[ f_{i,j}^{\lambda^i_j} : 0\leq i \leq k-1 , 1\leq j \leq \ell(\lambda^i)\right],
\]
with 
\begin{align*}
p(f_{i,j}) = i && \text{ and } && \rho(f_{i,j}) = 1.
\end{align*}
The multiplicity of $\chi^{\vec{\gamma}}$ is then the number of semistandard tableaux in the $f_{i,j}$, where $f_{i,j}$ appears only in $\gamma^i$, precisely $\lambda^i_j$ times.
Therefore
\[
\chi^{F^{\vec{\lambda}}} = \chi^{\vec{\lambda}} + \sum_{\vec{\gamma} \geq \vec{\lambda}} |\SSYT_k(\vec{\gamma}, F^{\vec{\lambda}})| ~ \chi^{\vec{\gamma}}, 
\]
where we write $\vec{\gamma} \geq \vec{\lambda}$ to mean that for every $i$, $\lambda^i \leq \gamma^i$ in the dominance (or lexicographic) order.
Therefore, these characters are triangularly related to irreducible characters, and every irreducible character can be written as a linear combination of the $\chi^{F^{\vec{\lambda}}}$. One can deduce that $\Frob(\chi^{F^{\vec{\lambda}}}) = h_{\vec{\lambda}}[X]$.

\section{Projective toric varieties and their characters}\label{section:semigroup-algebras}

This section will present the main application of our methods, and has two main objectives: 
\begin{enumerate}
    \item To calculate the bigraded Frobenius characteristic of the projective coordinate ring of a general Segre product of projective toric varieties (see Proposition~\ref{proposition:polytope-character}).
    
    \item To use this calculation, along with the theory of color rules developed in previous sections, to derive \emph{equivariant} versions of Euler-Mahonian identities (see Subsection \ref{subsection:euler-mahonian-identities} and Theorem \ref{theorem:equivariant-euler-mahonian}).
\end{enumerate}

To achieve the first objective, we must define these rings and endow them with the structure of a bigraded $Z_k \wr S_n$-module. We do this in Subsections \ref{subsection:semigroup-algebras-preliminaries} and \ref{subsection:actions-on-products}. Subsections \ref{subsection:equivariant-ehrhart-theory} and \ref{subsection:euler-mahonian-identities} are interludes which relate our research to the existing literature.

Finally, we achieve the second objective in Subsections \ref{subsection:products-of-projective-spaces} and \ref{subsection:proof-of-theorem-5-X} by restricting our attention to the product of projective spaces and applying our methods. We give these results a geometric interpretation as quotients by regular sequences which directly generalize the ones described in \cite{adeyemo-szendroi-refined} and \cite{braun-olsen-statistics}, some of them conjecturally.

\subsection{Preliminaries}\label{subsection:semigroup-algebras-preliminaries}
This subsection includes a very minimal summary of relevant facts about semigroup algebras and toric varieties. For a more detailed exposition, see \cite{miller-sturmfels-combinatorial} or \cite{cox-little-toric}.

\begin{note}
    From now until the end of the section, we use the following conventions, unless specified otherwise:
    \begin{enumerate}
        \item We work over the field of complex numbers $K = \CC$. In particular, we work with algebraic varieties over $K = \CC$.
        \item $P \subseteq \mathbb{R}^m$ is a full-dimensional convex lattice polytope. The lattice is $\ZZ^m$ except in Subsection \ref{subsection:equivariant-ehrhart-theory}, where it can be arbitrary.
    \end{enumerate}
\end{note}

For a non-negative integer $d \in \NN$, the set $dP = \{dx \in \RR^m: x \in P\}$ is the $d$-th \emph{dilation} of $P$. The \emph{cone} over $P$ is the set
\[
C(P) \coloneq \left \{ \sum_{i \in I} c_i (1, v_i) \in \RR^{m + 1}: \quad I \text{ finite, } c_i \in \RR_{\geq 0}, v_i \in P \text{ for all } i \in I \right \}.
\]
It is not hard to see that
\begin{equation}\label{eq:lattice-points-cone}
C(P) \cap \ZZ^{m + 1} = \coprod_{d \geq 0} \{d\} \times (dP \cap \ZZ^m).
\end{equation}
For a point $v \in \ZZ^m$, write $v = \sum_{i = 1}^m \pi_i(v)e_i$, where $\{e_i\}_{i = 1}^m$ is the canonical basis of $\RR^m$. The \emph{affine semigroup algebra} of $P$ is the set
\[
K[P] \coloneq K\left [ x_{d, v} \coloneq t^d\prod_{i =1}^m x_i^{\pi_i(v)}: \quad (d, v) \in C(P) \cap \ZZ^{m + 1} \right ] \subseteq K [t, x_1^{\pm 1}, \ldots, x_m^{\pm 1}].
\]
Equivalently, by (\ref{eq:lattice-points-cone}),
\begin{equation}\label{eq:affine-semigroup-algebra}
    K[P] = K \left [ x_{d, v} : \quad  d \in \NN, v \in dP \cap \ZZ^m \right ].
\end{equation}
It follows from a well-known result, known as \emph{Gordan's lemma}, that $K[P]$ is a finitely generated $K$-algebra. It admits a natural grading whose degree map is given by the (largest) exponent of $t$, which we call the \emph{projective degree}. We let $\pdeg(f)$ denote the projective degree of $f \in K[P]$.

Geometrically, $X_P \coloneq \Proj K[P]$ is a projective normal toric variety of (Krull) dimension $m$, and all such varieties arise in this way. \\

Let $S$ be a finite set of generators of $K[P]$, and let $T_P\coloneq K[S]$ be the $K$-algebra freely generated by the elements of $S$. The kernel $I_P := \ker(\pi)$ of the natural surjection
\begin{equation}\label{eq:natural-surjection}
    \pi: T_P \twoheadrightarrow K[P]
\end{equation}
is called the \emph{toric ideal} of $P$ and satisfies $K[P] \cong T_P/I_P$. We will presently endow $K[P]$ with multiple gradings, and this will be an isomorphism of multigraded $K$-algebras. Geometrically, $\Proj T_P$ is a weighted projective space, and the surjection $\pi$ induces an embedding
\[
X_P \xhookrightarrow{} \Proj T_P.
\]
Chapoton \cite{chapoton-analogues} and Adeyemo, Szendrői \cite{adeyemo-szendroi-refined} study a refinement of the grading on $K[P]$. Fix $\bm{a} \in \ZZ^m$, which we call a \emph{weight vector}. Let $d \in \NN$ and $v \in P \cap \ZZ^m$. The map
\[
\cdeg(x_{d, v}) := \bm{a} \cdot v \in \ZZ
\]
extends to $K[P]$ and endows it with an additional integral grading. Here, $\bm{a} \cdot v $ denotes the standard inner product of $\RR^m$. We will call this map the \emph{combinatorial degree}. It should be mentioned that Chapoton uses certain assumptions of positivity and genericity on the pair $(P, \bm{a})$ to develop his refined Ehrhart theory, but \cite{adeyemo-szendroi-refined} does not use it, and neither do we. It is also worth noting that we can get a multigrading on $K[P]$ in a similar way by choosing a finite collection $\{\bm{a}_i\}_{i = 1}^n \subseteq \ZZ^m$. \\

Another multigrading on $K[P]$ is studied by Reiner and Rhoades \cite{reiner-rhoades-harmonics}. Their multigrading is based on Macaulay inverse systems for Minkowski sums of point configurations and is different from the one we study, already in small examples (see their Remark 3.20). \\

It will often be useful to assume that $P$ is \emph{normal}, i. e., that for all $d \in \NN$ and $v \in dP \cap \ZZ^m$ there exist $v_1, v_2, \ldots, v_d \in P \cap \ZZ^m$, such that $v = \sum_{i = 1}^d v_i$. In particular, it can be shown that this is always the case if $P$ is $2$-dimensional. This notion of normality is also known as the \emph{integer decomposition property (IDP)}. Note that, when $P$ is normal, the affine semigroup algebra of $P$ is simply
\[
K[P] = K \left [ x_v \coloneq t\prod_{i = 1}^m x_i^{\pi_i(v)}: \quad v \in P \cap \ZZ^m \right ],
\]
and (\ref{eq:natural-surjection}) induces an embedding
\[
X_P \xhookrightarrow{} \Proj(T_P) \cong \PP^{|P \cap \ZZ^m| - 1}.
\]

\subsection{Actions on products}\label{subsection:actions-on-products}
If $P' \subseteq \mathbb{R}^{m'}$ is another full-dimensional convex lattice polytope, the Cartesian product
$P \times P'$ is again a polytope with these characteristics, and for all $d \in \NN$,
\begin{equation}\label{eq:dillation-product}
    d(P \times P') = dP \times dP'.
\end{equation}
This operation translates nicely to the geometric point of view. If $P$ and $P'$ are \emph{very ample} (a condition that is always satisfied by a dilation of $P$ and $P'$), then, by \cite[Theorem 2.4.7]{cox-little-toric}, the variety $X_{P \times P'} = \Proj K[P \times P']$ corresponds to the product of projective varieties $X_{P} \times X_{P'} = \Proj K[P] \times \Proj K[P']$ with the Segre embedding
\[
X_{P} \times X_{P'} \xhookrightarrow{} \mathbb{P}^{|P \cap \mathbb{Z}^m||P' \cap \mathbb{Z}^{m'}| - 1}.
\]
In particular, for $n \in \mathbb{N}_+$, we will be interested in $P^{\times n}$, the $n$-th Cartesian power of $P$. As in (\ref{eq:affine-semigroup-algebra}) (notice the additional subindices), and using (\ref{eq:dillation-product}), we may write
\begin{align*}
    K[P^{\times n}] &= K \left [ x_{d, v} \coloneq t^d \prod_{i = 1}^n \prod_{j = 1}^m x_{i, j}^{\pi_j(v_i)}: \quad d \in \NN, v_i \in dP \cap \mathbb{Z}^{m} \text{ for all } i = 1, \ldots, n \right ] \\
    &\subseteq K[t, x_{1, 1}^{\pm 1}, \ldots, x_{n, m}^{\pm 1}].
\end{align*}

The wreath product $Z_k \wr S_n$ acts linearly on $K[P^{\times n}]$ in a natural way, namely by permuting the lattice points of the $n$ copies of $P$ and multiplying by the corresponding roots of unity. More specifically, for $\sigma = u_{a_1}\sigma_1 \cdots u_{a_n}\sigma_n \in Z_k \wr S_n$, the action of $\sigma$ fixes $t$ and maps $x_{i, j}$ to $u_{a_i} x_{\sigma_i, j}$. Then, it is extended to $K[t, x_{1, 1}^{\pm 1}, \ldots, x_{n, m}^{\pm 1}]$ multiplicatively and linearly, and finally restricted to $K[P^{\times n}]$. This is a group action because if $\sigma' = u_{a_1}'\sigma_1' \cdots u_{a_n}'\sigma_n' \in Z_k \wr S_n$, then
\[
(\sigma'\sigma)(x_{i, j}) = u_{a_{\sigma_i}}'u_{a_i}x_{(\sigma'\sigma)_i, j} = \sigma'(u_{a_i} x_{\sigma_i, j}) = \sigma'(\sigma(x_{i, j})).
\]

In order to better understand this action, we need to give a couple of definitions.

\begin{definition}[Generalized permutation representations/matrices]\label{definition:permutation-representation}
Let $G$ be a group, $S$ a set, and $KS$ the $K$-vector space generated by the elements of $S$. We say $KS$ carries a \emph{generalized permutation representation} of $G$ if there is a representation
\[
\Pi: G \rightarrow \GL(KS)
\]
such that for all $g \in G$ and $s \in S$, there exist $u(g, s) \in K$ and $r \in S$ satisfying
\[
\Pi(g)s = u(g, s)r.
\]
In other words, the matrix $\Pi(g)$ in the \emph{standard basis} (the basis given by the elements of $S$) is a \emph{generalized permutation matrix} (a matrix having only one non-zero entry in each row and column). Note that, if $u(g,s) = 1$ for all $g \in G$ and $s \in S$, then $\Pi$ is a permutation representation and $\Pi(g)$ in the standard basis is a permutation matrix.
\end{definition}

The following proposition follows.

\begin{proposition}\label{proposition:Pi-is-generalised-permutation-representation}
    Let $d \in \NN$, and write $K[\cdot]_d$ for the $d$-th projective piece of $K[\cdot]$. The representation
    \[
    \Pi: Z_k \wr S_n \to \GL(K[t, x_{1, 1}^{\pm 1}, \ldots, x_{n, m}^{\pm 1}]_d)
    \]
    given by the linear action described above is a generalised permutation representation ($S$ is a monomial basis indexed by points in the lattice $\ZZ^{mn}$). Furthermore, in the standard basis, each matrix $\Pi(\sigma)$ is unitary.
\end{proposition}

To endow $K[P^{\times n}]$ with the structure of a \emph{multigraded} $Z_k \wr S_n$-module, we must impose conditions on the weight vectors $\bm{a} \in \ZZ^{mn}$ that induce the combinatorial grading.

\begin{proposition}\label{proposition:kpn-is-multigraded-zk-sn-module}
    The affine semigroup algebra $K[P^{\times n}]$ admits the structure of a multigraded $Z_k \wr S_n$-module, where the gradings are the projective degree and the combinatorial degrees are given by a finite collection of weight vectors $\{\bm{a}_i = (\bm{a}_i')^{\times n}\}_{i \in [I]} \subseteq \ZZ^{mn}$.
\end{proposition}

\begin{proof}
Let $\Pi$ be as in Proposition \ref{proposition:Pi-is-generalised-permutation-representation}. Note that the general case follows from the case in which there is only one weight vector $\bm{a} \in \ZZ^{mn}$. \\

Since the action fixes $t$, it leaves the projective degree invariant, so it suffices to show that it also leaves the combinatorial degree invariant. For $\sigma \in Z_k \wr S_n$, write $\tau = \sigma_1 \cdots \sigma_n$ for the underlying permutation in $S_n$. Let $d \in \NN$ and $x_{d, v} \in K[P^{\times n}]_d$. Since multiplying by a constant does not change the combinatorial degree:

\begin{equation}\label{eq:combinatorial-degree-only-depends-on-restriction}
    \cdeg(\sigma(x_{d, v})) = \cdeg(\Pi(\sigma)(x_{d, v})) = \cdeg(\Pi\downarrow^{Z_k \wr S_n}_{S_n}(\tau)(x_{d, v})).
\end{equation}

Let
\[
\Pi': S_n \to \GL_{mn}(\RR)
\]
be the representation defined by the following rule: if $\tau \in S_n$, and $\{e_{i, j} \in \RR^{mn}: i = 1, \ldots, m, ~j = 1, \ldots, n\}$ is the canonical basis of $\RR^{mn}$, then $\tau e_{i, j} = e_{\tau(i), j}$. We have that $\Pi' \cong m\Pi_{\defining}$, where $\Pi_{\defining}$ is the defining representation of $S_n$. From the definition of $x_{d, v}$, 
\begin{equation}\label{eq:restriction-is-permutation-representation}
    \Pi\downarrow^{Z_k \wr S_n}_{S_n}(\tau)(x_{d, v}) = x_{d, \Pi'(\tau)(v)}.
\end{equation}
In particular, the restriction $\Pi\downarrow^{Z_k \wr S_n}_{S_n}$ is a permutation representation.

The desired condition on the weight vector $\bm{a} \in \ZZ^{mn}$ is that they are $\Pi'(S_n)$-invariant. Indeed, let $\bm{a} \in \ZZ^{mn}$ be $\Pi'(S_n)$-invariant. Then,
\begin{align*}
    \cdeg(\sigma(x_{d, v})) 
    &= \cdeg(\Pi(\sigma)(x_{d, v})) \\
    &= \cdeg(\Pi\downarrow^{Z_k \wr S_n}_{S_n}(\tau)(x_{d, v})) \tag{\ref{eq:combinatorial-degree-only-depends-on-restriction}} \\
    &= \cdeg(x_{d, \Pi'(\tau)(v)}) \tag{\ref{eq:restriction-is-permutation-representation}} \\
    &= \bm{a} \cdot (\Pi'(\tau)(v)) \\
    &= (\Pi'(\tau^{-1})(\bm{a})) \cdot v \tag{$\Pi'$ is orthogonal in the canonical basis} \\
    &= \bm{a} \cdot v \tag{$a$ is $\Pi'(S_n)$-invariant} \\
    &= \cdeg(x_{d,v}).
\end{align*}

In defining $\Pi'$, we have ordered the canonical basis of $\RR^{mn}$ in the following way:
\[
e_{1, 1} \prec e_{1,2} \prec \cdots \prec e_{1, m} \prec e_{2, 1} \prec e_{2, 2} \prec \cdots \prec e_{2, m} \prec  \cdots \prec e_{n, m}.
\]
This implies that, if $\bm{a}$ is $\Pi'(S_n)$-invariant, then $\bm{a} = (\bm{a}')^{\times n}$, where $\bm{a}' \in \ZZ^{m}$. The result follows.
\end{proof}

\begin{note}
    We restrict ourselves to the bigraded case in Proposition \ref{proposition:kpn-is-multigraded-zk-sn-module} for the rest of the paper (i. e., the gradings are given by the projective degree and a single weight vector $\bm{a} = (\bm{a}')^{\times n} \in \ZZ^{mn}$).
\end{note}

We will use the following result extensively for the rest of this section:

\begin{proposition} \label{proposition:polytope-character}
    Consider the bigraded $Z_k \wr S_n$-module structure on $K[P^{\times n}]$ as in Proposition \ref{proposition:kpn-is-multigraded-zk-sn-module}. Then, for $\sigma = u_{a_1}\sigma_1 \cdots u_{a_n}\sigma_n \in Z_k \wr S_n$ of cycle type $\vec{\lambda}(\sigma) = \vec{\lambda}$, we have the following character:
    \[
    \chi^{K[P^{\times n}]}(\sigma) = \sum_{d \geq 0} t^d \prod_{i=0}^{k-1}
    \prod_{j = 1}^{\ell(\lambda^i)} \sum_{v \in dP \cap \mathbb{Z}^m} ~
    u_i^{|v|} q^{(\lambda^i_j) (\bm{a}' \cdot v)}.
    \]
\end{proposition}

\begin{proof}
    Fix $d \in \NN$. Let $K[P^{\times n}]_d$ be the $d$-th projective piece of $K[P^{\times n}]$. For $i = 1, \ldots, n$ and $v_i \in dP \cap \ZZ^m$, write
    \[
    x_{i, v_i} \coloneq \prod_{j = 1}^m x_{i, j}^{\pi_j(v_i)}.
    \]
    The set
    \[
    \mathcal{B}_d := \left \{ t^d \prod_{i = 1}^n x_{i, v_i}: \quad v_i \in dP \cap \ZZ^m, \text{ for all } i = 1, \ldots, n \right \}
    \]
    is the standard (monomial) basis of $K[P^{\times n}]_d$ (in particular, $|\mathcal{B}_d| = |dP^{\times n} \cap \ZZ^{mn}|$). Let $x = t^d \prod_{i = 1}^n x_{i, v_i} \in \mathcal{B}_d$. By definition,
    \begin{equation}\label{eq:image-basis-element}
        \sigma x = t^d \prod_{i = 1}^n u_{a_i}^{|v_i|} x_{\sigma_i, v_i},
    \end{equation}
    where $|v_i| = \sum_{j = 1}^m \pi_j(v_i)$. This gives a generalised permutation representation 
    \[
    \Pi_d: Z_k \wr S_n \to \GL(K\mathcal{B}_d).
    \]
    To calculate the coefficient of $t^d$ in the character $\chi^{K[P^{\times n}]}(\sigma)$ it suffices to count the fixed points $x_{\fixed}$ of the restriction $\Pi_d \downarrow_{S_n}^{Z_k \wr S_n}(\tau)$ with their corresponding roots of unity $u(x_{\fixed})$, where $\tau \in S_n$ is the permutation underlying $\sigma$. \\
    
    Let $D = \{(i, j) \in \ZZ^2: \quad 0 \leq i \leq k - 1, 1 \leq j \leq \ell(\lambda^i)\}$. It is not hard to see that these fixed points are the elements of $\mathcal{B}_d$ of the form
    \[
    x_{\fixed} = x_{\fixed}\left ((v_{ij})_{(i, j) \in D}\right ) = t^d \prod_{(i, j) \in D}\prod_{u_{a_k}\sigma_k \in c_j^i} x_{\sigma_{k}, v_{ij}}, \quad v_{ij} \in dP \cap \ZZ^m, \forall (i, j) \in D,
    \]
    where $c_j^i$ is the $C_i$-cycle of $\sigma$ indexed by $j$. In other words, to construct a fixed point $x_{\fixed}$, we must choose a point in $dP \cap \ZZ^m$ for each cycle $c_j^i$ of $\sigma$. \\
    
    By (\ref{eq:image-basis-element}), their corresponding roots of unity are
    \[
    u(x_{\fixed}) = \prod_{(i, j) \in D}\prod_{u_{a_k}\sigma_k \in c_j^i} u_{a_{\sigma_k}}^{|v_{ij}|} = \prod_{(i, j) \in D} u_i^{|v_{ij}|}.
    \]
    Their combinatorial degrees are
    \begin{align*}
        \cdeg(x_{\fixed}) 
        &= \cdeg \left ( \prod_{(i, j) \in D}\prod_{u_{a_k}\sigma_k \in c_j^i} x_{\sigma_{k}, v_{ij}} \right ) \\
        &= \sum_{(i, j) \in D} \sum_{u_{a_k}\sigma_k \in c_j^i} \cdeg(x_{\sigma_{k}, v_{ij}}) \\
        &= \sum_{(i, j) \in D} (\lambda_j^i) (\bm{a}' \cdot v_{ij}).
    \end{align*}
    Therefore,
    \begin{align*}
        [t^d]\chi^{K[P^{\times n}]}(\sigma) 
        &= \sum_{x_{\fixed}} u(x_{\fixed})q^{\cdeg(x_{\fixed})} \\
        &= \sum_{(v_{ij})_{i, j} \in (dP \cap \ZZ^m)^D} \prod_{(i, j) \in D} u_i^{|v_{ij}|}q^{(\lambda_j^i) (\bm{a}' \cdot v_{ij})} \\
        &= \prod_{(i, j) \in D} \sum_{v_{ij} \in dP \cap \ZZ^m} u_i^{|v_{ij}|}q^{(\lambda_j^i) (\bm{a}' \cdot v_{ij})}.
    \end{align*}
    The result follows.
\end{proof}

It is possible to find a plethystic expression for the character in Proposition~\ref{proposition:polytope-character} for any $k$ (see Theorem \ref{thm:projectivecharacter}). However, the story is specially nice for $k = 1$ (i. e., if the wreath product $Z_k \wr S_n$ is just $S_n$). Before giving the statement, however, we need to recall a few definitions from \cite{chapoton-analogues} or \cite{adeyemo-szendroi-refined}:

\begin{definition}[Refined Ehrhart polynomial, series]
    Let $\bm{a}' \in \ZZ^m$ be a weight vector. The \emph{refined Ehrhart polynomial} of $P$ is the Laurent polynomial
    \[
    L_{P, d}^{\bm{a}'}(q) = \sum_{i \in \ZZ} (\# \{v \in dP \cap \ZZ^m: \quad \bm{a}' \cdot v = i\}) q^i \in K[q, q^{-1}].
    \]
    The \emph{refined Ehrhart series} of $P$ is the formal power series
    \[
    E_P^{\bm{a}'}(t, q) = \sum_{d \geq 0} L_{P, d}^{\bm{a}'}(q)t^d \in K[q, q^{-1}][\![ t ]\!].
    \]
\end{definition}

These definitions are justified as their specialisations at $q = 1$ are the classical Ehrhart polynomial and series, respectively. In \cite{adeyemo-szendroi-refined}, explicit expressions are given in the case of the simplex, cross polytope, and hypercube. 

\begin{corollary} \label{corollary:frobenius-characteristic}
    Consider the natural bigraded $S_n$-module structure on $K[P^{\times n}]$ as in Proposition \ref{proposition:kpn-is-multigraded-zk-sn-module} with $k = 1$. Then, the bigraded Frobenius characteristic satisfies
    \[
    \Frob(K[P^{\times n}]) = \sum_{d \geq 0} t^d \sum_{\lambda \vdash n} s_\lambda \left [ L_{P, d}^{\bm{a}'}(q) \right ] s_\lambda[X],
    \]
    where $L_{P, d}^{\bm{a}'}(t)$ is the refined Ehrhart polynomial of $P$.
\end{corollary}

\begin{proof}
    Proposition~\ref{proposition:polytope-character} implies the following equalities for $k = 1$:
    \begin{align*}
       \Frob(K[P^{\times n}])
        &= \frac{1}{n!}\sum_{\sigma \in S_n}  \chi^{K[P^{\times n}]}(\sigma) p_{\lambda(\sigma)}[X] \\
        &= \frac{1}{n!} \sum_{\sigma \in S_n} \sum_{d \geq 0} t^d
        \prod_{j = 1}^{\ell(\lambda(\sigma))} \sum_{v \in dP \cap \mathbb{Z}^m} ~
        q^{(\lambda_j) (\bm{a}' \cdot v)} p_{\lambda(\sigma)}[X] \\
        &= \frac{1}{n!} \sum_{\sigma \in S_n} \sum_{d \geq 0} t^d
        p_{\lambda(\sigma)} \left [ \sum_{v \in dP \cap \mathbb{Z}^m} ~
        q^{(\bm{a}' \cdot v)} X \right ] \\
        &= \sum_{d \geq 0} t^d \sum_{\lambda \vdash n} s_\lambda \left [ \sum_{v \in dP \cap \mathbb{Z}^m} ~
        q^{(\bm{a}' \cdot v)} \right ] s_\lambda[X] \\
        &= \sum_{d \geq 0} t^d \sum_{\lambda \vdash n} s_\lambda \left [ L_{P, d}^{\bm{a}'}(q) \right ] s_\lambda[X].
    \end{align*}
\end{proof}

\begin{example}\label{example:cross-polytope}
Let $P = \Xi_m \subseteq \RR^m$ be the $m$-dimensional cross-polytope, defined as the convex hull of the set
\[
\{e_1, -e_1, e_2, -e_2, \ldots, e_m, -e_m\}.
\]
By definition,
\begin{align*}
    K[(\Xi_m)^{\times n}]
    &= K\left [ t \prod_{i \in I} x_{i, j_i}^{\pm 1}: \quad I \subseteq \{1, \ldots, n\}, j_i \in \{1, \ldots, m\} \text{ for all } i \in I \right ] \\
    &\subseteq K \left [t, x_{1, 1}^{\pm 1}, \ldots, x_{n, m}^{\pm1} \right ].
\end{align*}

For example,
\begin{align*}
    K[(\Xi_2)^{\times 2}] 
    &= K\left [ \begin{matrix}
    t,   &     &     &     \\
    tx_{1, 1}, & tx_{1, 1}^{-1}, & tx_{1, 2}, & tx_{1, 2}^{-1}, \\
    tx_{2, 1}, & tx_{2, 1}^{-1}, & tx_{2, 2}, & tx_{2, 2}^{-1}, \\
    tx_{1, 1}x_{2, 1}, & tx_{1, 1}x_{2, 1}^{-1}, & tx_{1, 1}x_{2, 2}, & tx_{1, 1}x_{2, 2}^{-1}, \\
    tx_{1, 1}^{-1}x_{2, 1}, & tx_{1, 1}^{-1}x_{2, 1}^{-1}, & tx_{1, 1}^{-1}x_{2, 2}, & tx_{1, 1}^{-1}x_{2, 2}^{-1}, \\
    tx_{1, 2}x_{2, 1}, & tx_{1, 2}x_{2, 1}^{-1}, & tx_{1, 2}x_{2, 2}, & tx_{1, 2}x_{2, 2}^{-1}, \\
    tx_{1, 2}^{-1}x_{2, 1}, & tx_{1, 2}^{-1}x_{2, 1}^{-1}, & tx_{1, 2}^{-1}x_{2, 2}, & tx_{1, 2}^{-1}x_{2, 2}^{-1}
    \end{matrix} \right ] \\
    &\subseteq K \left [t, x_{1, 1}, x_{1, 1}^{-1}, x_{1, 2}, x_{1, 2}^{-1}, x_{2, 1}, x_{2, 1}^{-1}, x_{2, 2}, x_{2, 2}^{-1} \right ].
\end{align*}
If the weight vector is $\bm{a}' = (1, 1)$, then
\[
L^{\bm{a}'}_{\Xi_2,d}(q)  = L^{\bm{a}'}_{\Xi_2,d-1}(q) + (d+1)( q^{-d} +q^d) + 2\left( q^{-d+2} +q^{-d+4}  + \cdots + q^{d-2} \right) .
\]
It follows that 
\begin{align*} E^{\bm{a}'}_{\Xi_2}(t,q) = 
\sum_{d \geq 0} t^dL^{\bm{a}'}_{\Xi_2,d}(q) &= \frac{
(1-t^2)(1+t) 
}{
(1-qt)^2 (1-t/q)^2
} \\
& = \sum_{d \geq 0} t^d h_d\left[ 2 \left(q+\frac{1}{q} - \epsilon \right) -1 \right],
\end{align*}
where $\epsilon$ is treated as a variable within the plethystic brackets then evaluated at $-1$ outside of the brackets.
Therefore,
\begin{align*}
    \Frob(K[X_{\Xi_2}^{\times 2}])(X) 
    &= \sum_{d\geq 0 } t^d 
    \sum_{\lambda \vdash 2} s_\lambda \left[ h_d\left[ 2 \left(q+\frac{1}{q} - \epsilon \right) -1 \right]\right] s_\lambda[X].
\end{align*}
The Schur coefficients can then be calculated. For instance, since $s_{1,1} = p_{1}^2/2 -p_2/2 ,$ the coefficient of $s_{1,1}$ is given by
\begin{align*} &
\frac{1}{2} \sum_{d \geq 0 } t^d
\left( h_d\left[ 2 \left(q+\frac{1}{q} - \epsilon \right) -1 \right]^2 - 
 h_d\left[ 2 \left(q^2+\frac{1}{q^2} - \epsilon \right) -1 \right]
\right) \\
 & \hspace{4em} =t \left( q^2+2q+6+\frac{2}{q}+\frac{1}{q^2} \right) \\
  & \hspace{8em}+ 
 t^2 \left(q^4+6 q^3+14 q^2+12 q+14+ \frac{12}{q}+ \frac{14}{q^2}+ \frac{6}{q^3} + \frac{3}{q^4} \right) + \cdots.
\end{align*}
\end{example}

\begin{example}\label{example:product-of-unit-simplices}
    Let $P = \Delta_m \subseteq \RR^m$ be the $m$-dimensional unit simplex, defined as the convex hull of the set
    \[
    \{0, e_1, e_2, \ldots, e_m\}.
    \]
    By definition,
    \begin{align*}
    K[(\Delta_m)^{\times n}] 
    &= K \left [ t \prod_{i \in I} x_{i, j_i}: \quad I \subseteq \{1, \ldots, n\}, j_i \in \{1, \ldots, m\} \text{ for all } i \in I \right ] \\
    &\subseteq K[t, x_{1, 1}, \ldots, x_{n, m}].
    \end{align*}
    If the weight vector is $\bm{a}' = (1, 2, \ldots, m)$, \cite{adeyemo-szendroi-refined} shows that
    \[
    L_{\Delta_m, i}^{(1, 2, \ldots, m)} = \qbinom{i + m}{m}_q,
    \]
    where $\qbinom{i + m}{m}_q$ is a $q$-binomial coefficient. Therefore, Corollary \ref{corollary:frobenius-characteristic} implies that
    \[
    \Frob(K[\Delta_m^{\times n}]) = \sum_{i \geq 0} t^i \sum_{\lambda \vdash n} s_\lambda \left [ \qbinom{i + m}{m}_q   \right ] s_\lambda[X].
    \]
    We will refer to this example repeatedly in Subsection \ref{subsection:products-of-projective-spaces}.
\end{example}
The following application of Theorem~\ref{theorem:mainresult} gives a combinatorial description of the multiplicities of irreducible representations of $Z_k \wr S_n$ appearing in $K[P^{\times n}]$ for any $P$. We will use this for Theorem~\ref{theorem:equivariant-euler-mahonian}, where we give explicit combinatorially-defined statistics counting multiplicities in the quotient of $K[P^{\times n}]$ by a regular sequence when $P = \Delta_1 = [0, 1] \subseteq \RR$.

\begin{theorem}\label{thm:projectivecharacter}
    For $d \in \NN$, let $F_d = [ v^n: v \in dP \cap \ZZ^m]$ be a color rule with value $p(v) = |v|$ and weight $\rho(v) = q^{\bm{a}' \cdot v}$. Let $\chi^{F_d}$ be the character given by this color rule. Then,
    \[
    \chi^{K[P^{\times n}]} = \sum_{d \geq 0} t^d \chi^{K[P^{\times n}]_d} = \sum_{d \geq 0} t^d \chi^{F_d}.
    \]
    In particular, for any $\vec{\gamma} \vdash_k n$,
    \[
    \langle \chi^{K[P^{\times n}]} , \chi^{\vec{\gamma}} \rangle
    = \sum_{d \geq 0} t^d \sum_{T \in \SSYT_k(\vec{\gamma},F_d)} \rho(T).
    \]
\end{theorem}

\begin{proof}
    The first statement follows from  follows from Proposition~\ref{proposition:polytope-character}. The second statement follows from the first and Theorem~\ref{theorem:mainresult}.
\end{proof}

\subsection{Interlude 1: Equivariant Ehrhart theory}\label{subsection:equivariant-ehrhart-theory}

\emph{Equivariant Ehrhart theory} was introduced by Stapledon \cite{stapledon-equivariant} with motivations from algebraic geometry and mirror symmetry. Its main object of study is the \emph{equivariant Ehrhart series} which arises when a finite group $G$ acts linearly on the lattice points of a $G$-invariant convex lattice polytope $P$. We recommend \cite{elia-kim-supina-techniques} as a reference, where the theory is further developed. \\

Corollary \ref{corollary:frobenius-characteristic} may be seen as a statement in a \emph{refined} version of this equivariant Ehrhart theory. The goal of this subsection is to make this relation to the existing literature explicit. The first step is to describe a setup that is equivalent to the one in \cite{elia-kim-supina-techniques}.
\begin{definition}[Equivariant Ehrhart series]\label{definition:equivariant-ehrhart-series}
    Assume the following setup:
    \begin{enumerate}
        \item $L \subseteq \RR^N$ is a lattice (i. e., a free abelian subgroup) of rank $N$.
        \item $G$ is a finite group.
        \item $\Pi': G \to \GL(L)$ is a lattice representation (i. e., a representation $\Pi': G \to \GL_N(\RR)$ such that $L$ is $\Pi'(G)$-invariant).
        \item $P \subseteq \RR^N$ is a full-dimensional convex lattice polytope that is $\Pi'(G)$-invariant.
        \item For $d \in \NN$, $\Pi_{P, d}: G \to \GL(K(dP \cap L))$ is the permutation representation carried by the vector space $K(dP \cap L)$ such that, for all $g \in G$ and $x \in dP \cap L$, 
        \[ \Pi_{P, d}(g)(x) = \Pi'(g)(x). \]
        \item $\Pi_P = \bigoplus_{d \in \NN} \Pi_{P, d}: G \to \GL(K(C(P) \cap (L \times \ZZ)))$. In particular, $\Pi_P$ is a graded representation.
    \end{enumerate}
    Then, the \emph{equivariant Ehrhart series} of $P$ is the formal power series
    \[
    EE_P(t) = \sum_{d \geq 0} \chi(\Pi_P)_d t^d = \sum_{d \geq 0} \chi(\Pi_{P, d}) t^d \in R(G)[\![ t ]\!],
    \]
    where $R(G)$ is the ring of $K$-valued class functions of $G$.
\end{definition}

To refine this series we have to endow $\Pi_P$ with a bigrading.

\begin{proposition}\label{proposition:rho-prime-p-is-bigraded-representation}
    Assume the setup in Definition \ref{definition:equivariant-ehrhart-series}. Additionally:
    \begin{enumerate}
    \setcounter{enumi}{6}
        \item $\Pi'$ is orthogonal with respect to the standard inner product of $\RR^N$ (this choice of inner product can be made without loss of generality).
        \item $\Pi'$ contains the trivial representation $\Pi_{\triv}$ of $G$ with multiplicity $m_{\triv} \geq 1$. Equivalently, $\Pi'(G)$ acts trivially on a subspace $V_{\triv} \subseteq \RR^N$ of dimension $\geq 1$.
        \item $\bm{a} \in V_{\triv} \cap L$,
        \item $\ZZ_{\bm{a}, L} \coloneq \bm{a} \cdot L$, where $\cdot$ denotes the standard inner product of $\RR^N$.
    \end{enumerate}
    Then, we have a decomposition
    \[
    \Pi_P = \bigoplus_{d \in \NN} \bigoplus_{i \in \ZZ_{\bm{a}, L}} \Pi_{P, d, i},
    \]
    where, for each $i \in \ZZ_{\bm{a}, L}$ and $g \in G$, $\Pi_{P, d, i}(g)$ is the restriction of $\Pi_{P, d}(g)$ to the subspace $K \{v \in dP \cap L: ~~ \bm{a} \cdot v = i\}$.
   In particular, $\Pi_P$ is a \emph{bigraded representation}.
\end{proposition}
\begin{proof}
    It suffices to prove that $K \{v \in dP \cap L: ~~ \bm{a} \cdot v = i\}$ is a submodule of $K(dP \cap L)$. Let $g \in G, d \in \NN, i \in \ZZ_{\bm{a}, L}$, and $v \in dP \cap L$ such that $\bm{a} \cdot v = i$. Then,
    \begin{align*}
    \bm{a} \cdot (\Pi_{P, d}(g)(v)) 
        &=\bm{a} \cdot (\Pi'(g)(v)) \\
        &=(\Pi'(g^{-1})(\bm{a})) \cdot v \tag{$\Pi'$ is orthogonal} \\
        &=\bm{a} \cdot v \tag{$\bm{a} \in V_{\triv}$} \\
        &=i.
    \end{align*}
    This proves the claim.
\end{proof}
From now on, we write \emph{REET} for the setup described in Definition \ref{definition:equivariant-ehrhart-series} and Proposition \ref{proposition:rho-prime-p-is-bigraded-representation}. We may now define the refinement.
\begin{definition}[Refined equivariant Ehrhart series]
    Assume REET. Then, the \emph{refined equivariant Ehrhart series} of $P$ is the formal power series
    \[
    EE_P(t, q) = \sum_{d \geq 0} \sum_{i \in \ZZ_{\bm{a}, L}} \chi(\Pi_P)_{d, i} t^d q^i = \sum_{d \geq 0} \sum_{i \in \ZZ_{\bm{a}, L}} \chi(\Pi_{P, d, i}) t^dq^i \in R(G)[q, q^{-1}][\![ t ]\!].
    \]
\end{definition}
\begin{remark}
    Assume REET. We have
    \[
    K[P] \cong K(C(P) \cap (L \times \ZZ))
    \]
    as bigraded $G$-modules.
\end{remark}
If $G = S_n$, we can relate the bigraded Ehrhart series to the bigraded Frobenius character:
\begin{corollary}\label{corollary:relation-frob-ee}
    Assume REET. Let $n \in \NN_+$ and $G = S_n$. Then, the bigraded Frobenius characteristic satisfies
    \[
    \Frob(K[P])  = \Frob(K(C(P) \cap (L \times\ZZ)))  = \frac{1}{n!}\sum_{\sigma \in S_n} EE_P(t, q)(\sigma) p_{\lambda(\sigma)}[X],
    \]
    where $EE_P(t, q)(\sigma)$ is the refined equivariant Ehrhart series of $P$ evaluated at $\sigma$.
\end{corollary}

We have achieved our goal, since Corollary \ref{corollary:frobenius-characteristic} gives a more explicit expression for the character in Corollary \ref{corollary:relation-frob-ee} in the following particular case:

\begin{itemize}
    \item $L = \ZZ^{mn}$ ($N = mn$),
    \item $P = (P')^{\times n} \subseteq \RR^{mn}$, where $P' \subseteq \RR^m$ is a full-dimensional convex lattice polytope, and
    \item $\Pi'$ decomposes (up to reordering coordinates so that $P$ is $\Pi'(S_n)$-invariant) as $m\Pi_{\defining}$, where $\Pi_{\defining}$ is the defining representation of $S_n$. This implies that $\bm{a} = (\bm{a}')^{\times n}$, where $\bm{a}' \in \ZZ^m$.
\end{itemize}

Note that, even in this particular case, not all $\Pi'(S_n)$-invariant polytopes $P \subseteq \RR^{mn}$ are Cartesian products $(P')^{\times n} \subseteq \RR^{mn}$, where $P' \subseteq \RR^m$. It would be very interesting to find refined analogues of other results and techniques in \cite{elia-kim-supina-techniques} for polytopes which are not Cartesian products. We are currently working in this direction. \\

Also note that if $k \neq 1$, the linear action we consider in Subsection \ref{subsection:actions-on-products} is \emph{not} a permutation representation. However, most of the statements in this subsection still hold after slight modifications.

\subsection{Interlude 2: Euler-Mahonian identities and the projective coinvariant algebra}\label{subsection:euler-mahonian-identities}

Before continuing, we would like to put the results of this section into context with the current literature. We will use Tables \ref{tab:euler-mahonian-identities} and \ref{tab:equivariant-euler-mahonian-identities} as an aid. The graded rings $R_{m, n, k}$ in the first column will be defined in Subsection \ref{subsection:products-of-projective-spaces}.
\begin{table}[!ht]
\centering
\small
\begin{tabular}{|p{15mm}|l|p{20mm}|p{45mm}|}
\hline
\textbf{Graded ring} & \textbf{Identity} & \textbf{Reference} & \textbf{Notes} \\
\hline
$\rule{0em}{1.6em} R_n$ & $ \frac{1}{(1-q)^n} = \frac{ \sum_{\sigma \in S_n} q^{\maj(\sigma)}}{\prod_{j=1}^n (1-q^j)}$ & MacMahon \cite{MacMahon} & $\maj$ is the major index. \\
\hline
$\rule{0em}{1.6em} R_{1, n, 1}$ & $\sum_{i \geq 0} ([i + 1]_q)^n t^i = \frac{\sum_{\sigma \in S_n} t^{\des(\sigma)}q^{\maj(\sigma)}}{\prod_{j = 0}^n (1-tq^j)}$ & Carlitz \cite{carlitz-combinatorial} & $\des$ is the number of descents.\\
\hline
$\rule{0em}{1.6em} R_{1, n, k}$ & $\sum_{i \geq 0} ([i + 1]_q)^n t^i = \frac{\sum_{\sigma \in Z_k \wr S_n} t^{\ndes(\sigma)}q^{\nmaj(\sigma)}}{(1 - t)\prod_{j = 1}^n (1-t^kq^{kj})}$ & Bagno \cite{bagno-euler} & $\ndes$ is the number of negative descents, $\nmaj$ is the negative major index.\\
& $\quad\quad\quad\quad\quad\quad\quad\quad = \frac{\sum_{\sigma \in Z_k \wr S_n} t^{\fdes(\sigma)}q^{\fmaj(\sigma)}}{(1 - t)\prod_{j = 1}^n (1-t^kq^{kj})}$ & Bagno, Biagioli \cite{bagno-biagioli-colored} & $\fdes$ is the number of flag descents, $\fmaj$ is the flag major index. \\
\hline
\end{tabular}
\caption{Euler-Mahonian identities.} \normalsize{Taking the Hilbert series of the graded rings in the first column gives the numerators in the identities from the second column.}
\label{tab:euler-mahonian-identities}
\end{table}

Braun and Olsen \cite{braun-olsen-statistics} call the formulas in Table \ref{tab:euler-mahonian-identities} \emph{Euler-Mahonian identities.} These formulas have appeared in many different contexts within algebraic combinatorics (see the discussion after Theorem 1.1 in \cite{braun-olsen-statistics}). In their paper, they study quotients of the form
\[
R_{n, k} \coloneq K[(\Delta_1)^{\times n}]/(J_{k, n}),
\]
where $K[(\Delta_1)^{\times n}]$ is the $n$-dimensional unit hypercube, and $J_{k, n}$ is an ideal generated by a regular sequence which is invariant under a linear action of the wreath product $Z_k \wr S_n$ on $K[(\Delta_1)^{\times n}]$. They define a grading on $K[(\Delta_1)^{\times n}]$ such that $J_{k, n}$ is homogeneous, and obtain a Gr\"{o}bner basis for $R_{n, k}$. Then, they interpret the identities by Carlitz, Bagno, and Bagno, Biagioli in Table \ref{tab:euler-mahonian-identities} as the bigraded Hilbert series of $R_{n, 1}$ and $R_{n, k}$, respectively. Finally, in Section 6, they also hint at a relation between $R_{n, 1}$ and the classical \emph{coinvariant algebra}
\begin{equation}\label{equation:coinvariant-algebra}
    R_n \coloneq K[x_1, x_2, \ldots, x_n]/K[x_1, x_2, \ldots, x_n]_+^{S_n},
\end{equation}
where $K[x_1, x_2, \ldots, x_n]_+^{S_n}$ is the ideal generated by all homogeneous invariants of positive degree. Nevertheless, the idea of using coinvariant algebras such as (\ref{equation:coinvariant-algebra}) to study multivariate statistics is much older, and has been used successfully in \cite{adin-brenti-roichman-descent}, for example, where it is credited to Ira Gessel. 
\begin{remark}
    The linear action studied in \cite{braun-olsen-statistics} is exactly the action we study in Subsection \ref{subsection:actions-on-products} in the special case when $P = (\Delta_1)^{\times n}$, and their bigrading is given, in our language, by the projective degree and the weight vector $\bm{a} = (1)^{\times n}$. Moreover, their ideals $J_{k, n}$ are exactly the ideals $I_{1, n, k}$ we will define after our Conjecture \ref{conjecture:regular-sequence}, and their quotients $R_{k, n}$ are the quotients $R_{1, k, n}$.
\end{remark}
\begin{table}[H]
\centering
\small
\begin{tabular}{ | p{15mm} | l | p{32mm} |p{35mm}|}
\hline
\textbf{Graded ring} & \textbf{Identity} & \textbf{Reference} & \textbf{Notes} \\
\hline
$\rule{0pt}{3.5em} R_n$ &
$\begin{aligned} &\sum_{\lambda \vdash n} s_{\lambda}[X] s_{\lambda}\left[ \frac{1}{1-q}\right]
\\
&= \frac{ \sum_{\lambda \vdash n} \sum_{T \in \SYT(\lambda)} q^{\maj(T)} s_\lambda[X]}{\prod_{j=1}^n (1-q^j)}
\end{aligned}$ 
& Lusztig, Stanley & Equivariant version of MacMahon's identity. \\[1.5em]
\hline
$\rule{0em}{3.5em} R_{1, n, 1}$&
$\begin{aligned}
    &\sum_{i \geq 0} t^i \sum_{\lambda \vdash n} s_\lambda \left [ [i + 1]_q \right ] s_\lambda[X] \\
    &= \frac{\sum_{\lambda \vdash n} \sum_{T \in \SYT(\lambda)} t^{\des(T)} q^{\comaj(T)} s_\lambda[X]}{\prod_{j = 0}^n (1 - tq^j)}
\end{aligned}$
& Szendrői~\cite{szendroi-projective}, Raicu, Sam, Weyman~\cite{raicu-sam-weyman-modules}, Theorem~\ref{theorem:equivariant-euler-mahonian} & Equivariant version of Carlitz's identity. \\[.4em]
\hline
$\rule{0em}{3.8em} R_{1, n, k}$ & 
$\begin{aligned}
    &\sum_{r \geq 0} t^r \sum_{\vec{\lambda} \vdash_k n} \prod_{i=0}^{k-1} \prod_{j = 1}^{\ell(\lambda^i)} [r + 1]_{q_{i, j}} \overline{P}_{\vec{\lambda}}[X]/Z_{\vec{\lambda}} \\
    &= \frac{\sum_{\vec{\lambda} \vdash_k n}\sum_{T \in \SYT(\vec{\lambda})} t^{{\wdes}(T)} q^{{\wcomaj}(T)} s_{\vec{\lambda}}[X]}{(1 - t)\prod_{j = 1}^n (1 - t^kq^{kj})}
\end{aligned}$
& Theorem \ref{theorem:equivariant-euler-mahonian} & Equivariant version of Bagno's and Bagno, Biagioli's identities. \\[.4em]
\hline 
$\rule{0em}{3.5em} R_{m, n, 1}$ & $\begin{aligned}
    &\sum_{i \geq 0} t^i \sum_{\lambda \vdash n} s_\lambda \left [ \qbinom{i + m}{m}_q   \right ] s_\lambda[X] \\
    &= \frac{\sum_{\lambda \vdash n} h_{S_n, \lambda}(t, q) s_\lambda[X]}{\prod_{j = 0}^{mn} (1 - tq^j)}
\end{aligned}$ & Proposition~\ref{proposition:equivariant-euler-mahonian-non-explicit},~1. & $h_{S_n, \lambda} \in \NN[t, q]$. The~statistics~are~not necessarily Euler-Mahonian. \\[.4 em]
\hline
$\rule{0em}{2.8em} R_{m, n, k}$ & $\begin{aligned}
    &\Frob\left(K[(\Delta_m)^{\times n}]\right) \\
    &= \frac{\sum_{\vec{\lambda} \vdash_k n} h_{Z_k \wr S_n, \vec{\lambda}}(t, q) s_{\vec{\lambda}}[X]}{(1 - t)\prod_{j = 1}^{mn} (1 - t^kq^{kj})}.
\end{aligned}$ & Proposition~\ref{proposition:equivariant-euler-mahonian-non-explicit},~2. & $h_{Z_k \wr S_n, \vec{\lambda}} \in \NN[t, q]$ (conjecturally for large $k$). The~statistics~are~not necessarily Euler-Mahonian. \\[.4 em]
\hline
\end{tabular}
\caption{Equivariant Euler-Mahonian identities.} \normalsize{Taking the Frobenius characteristic of the graded $S_n$ or $Z_k \wr S_n$-modules in the first column gives the numerators in the identities from the second column.}
\label{tab:equivariant-euler-mahonian-identities}
\end{table}

Since all graded rings in Table \ref{tab:euler-mahonian-identities} admit the structure of graded $S_n$ or $Z_k \wr S_n$-modules, it is natural to take their Frobenius characteristic. We call the resulting formulas \emph{equivariant} Euler-Mahonian identities (see Table \ref{tab:equivariant-euler-mahonian-identities}). For example, the identity for $R_n$ is a classical result by Lusztig, Stanley. In a privately shared preprint \cite{szendroi-projective}, Szendrői notes that the relation between $R_{n, 1}$ and $R_n$ can be exploited to refine this identity, as in Table \ref{tab:equivariant-euler-mahonian-identities}, row 2. For this reason, he refers to $R_{n, 1}$ as the \emph{projective coinvariant algebra}. This was, in fact, the starting point of our work in this section.

Raicu, Sam, Weyman \cite{raicu-sam-weyman-modules} also prove the identity in Table \ref{tab:equivariant-euler-mahonian-identities}, row 2 while studying modules supported on Chow varieties. In the proof of Theorem \ref{theorem:equivariant-euler-mahonian}, we sketch a proof using classical results on $q$-binomial coefficients. It also follows from the identity in Table \ref{tab:equivariant-euler-mahonian-identities}, row 3. 

The last three rows of Table \ref{tab:equivariant-euler-mahonian-identities} are new results we will see later in this section.

\subsection{Products of projective spaces}\label{subsection:products-of-projective-spaces}
We now resume the discussion we started in Subsection \ref{subsection:actions-on-products}. The wreath product action and bigraded algebras we have defined so far exhibit very interesting behaviour. For example, it is evident, from the formulas in Proposition~\ref{proposition:polytope-character} or Corollary \ref{corollary:frobenius-characteristic}, that the characters of these algebras feature negative exponents of $q$. This is also the case in Example \ref{example:cross-polytope}. \\

If $P = \Delta_m \subseteq \mathbb{R}^m$, the $m$-dimensional unit simplex (see Example \ref{example:product-of-unit-simplices}), then the corresponding product variety is $(\mathbb{P}^m)^{\times n}$ with the Segre embedding. Recall that, by a theorem of Hochster \cite{hochster-rings}, the affine semigroup algebra of a (normal) polytope is Cohen-Macaulay. Results in \cite{adeyemo-szendroi-refined} and \cite{braun-olsen-statistics}, along with computational experiments, suggest the following:

\begin{conjecture}\label{conjecture:regular-sequence}
    Let $m, n, k \in \NN_+, P = \Delta_m \subseteq \mathbb{R}^m$, and $\bm{a}' = (1, 2,\ldots, m)$. Then, the elements
    \[
    e^0 = t, \quad e^i_{m, n, k} = \sum_{\substack{x \in K[(\Delta_m)^{\times n}]\\x \text{ monic monomial}\\\pdeg(x) = 1, \cdeg(x) = i}} x^k, i = 1, \ldots, mn
    \]
    form a regular sequence in $K[(\Delta_m)^{\times n}]$ of maximal length. 
\end{conjecture}

Define $I_{m, n, k}$ to be the ideal generated by the regular sequence in Conjecture \ref{conjecture:regular-sequence}. Note that $I_{m, n, k}$ is invariant under the linear action of $Z_k \wr S_n$ and homogeneous with respect to the bigrading. Define
\[
R_{m, n, k} \coloneq K[P^{\times n}] / I_{m, n, k}
\]
for the corresponding quotient. If Conjecture \ref{conjecture:regular-sequence} holds, $R_{m, n, k}$ is a finite-dimensional $K$-vector space.

The statement and idea of the proof of point 2 in the following proposition was hinted to us by Balázs Szendrői.

\begin{proposition}\label{proposition:regular-sequence}
    Conjecture \ref{conjecture:regular-sequence} is true in the following cases:
    \begin{enumerate}
        \item $k = m  = 1 \text{ and } n \in \NN_+$;
        \item $k = 1 \text{ and } m, n \in \NN_+$;
        \item $m = 1 \text{ and } k, n \in \NN_+$;
        \item \begin{enumerate}[i.]
            \item $m = 2$ and $n, k \leq 10$;
            \item $m = 3$, $n = 2$ and $k \leq 7$;
            \item $m = 4$, $n = 2$ and $k \leq 3$.
        \end{enumerate}
    \end{enumerate}
\end{proposition}

\begin{proof}
\hfill
    \begin{enumerate}
        \item For a geometric proof, see \cite{adeyemo-szendroi-refined}, Proposition 2.7. Alternatively, this point follows from point 3, and we will see it is equivalent to point 2.
        
        \item We will give an isomorphism
        \[
        \varphi: K[(\Delta_1)^{\times mn}]^{S_{(m^n)}} \xrightarrow{\sim} K[(\Delta_m)^{\times n}].
        \]
        Here
        \[
        S_{(m^n)} = S_{I_1} \times S_{I_2} \times \cdots \times S_{I_n}
        \]
        is the standard \emph{Young subgroup} of $S_{mn}$, where, for $i \in \{1, \ldots, n\}$,
        \[
        I_i \coloneq \{(i - 1)m+1, (i - 1)m + 2, \ldots, im\}.
        \]
        
        We now describe the isomorphism $\varphi$. The relevant algebras are
        \begin{align*}
            K[(\Delta_1)^{\times mn}]^{S_{(m^n)}} &= K\left [t\prod_{i \in I}x_{i, 1}: \quad I \subseteq \{1,\dots,mn\} \right]^{S_{(m^n)}} \\
            &\subseteq K\left [t\prod_{i \in I}x_{i, 1}: \quad I \subseteq \{1,\dots,mn\} \right] \\
            &\subseteq K[t, x_{1,1},x_{2,1},\ldots,x_{mn, 1}]
        \end{align*}
        and
        \begin{align*}
        K[(\Delta_m)^{\times n}] &= K \left [ t \prod_{i \in I} x_{i, j_i}: \quad I \subseteq \{1, \ldots, n\}, j_i \in \{1, \ldots, m\} \text{ for all } i \in I \right ] \\
        &\subseteq K[t,x_{1, 1},x_{1,2},\ldots,x_{1,m},x_{2,1},x_{2,2},\ldots,x_{2,m},\ldots,x_{n,m}]
        \end{align*}
        (see Example \ref{example:product-of-unit-simplices}). We then see that $K[(\Delta_1)^{\times mn}]^{S_{(m^n)}}$ is generated as a $K$-algebra by elements of the form
        \[
        x(j_1, j_2, \ldots, j_n) = t\left ( \prod_{i = 1}^n\sum_{\substack{J_i \subseteq I_i\\|J_i| = j_i}}\prod_{k \in J_i} x_{k,1} \right ),
        \]
        where $0 \leq j_i \leq |I_i| = m$ for all $i = 1, \ldots, n$. The isomorphism $\varphi$ is now clear:
        \[
        x(j_1, j_2 \ldots, j_n) \mapsto t\prod_{i = 1}^nx_{i,j_i}.
        \]
        It is not hard, although a bit tedious, to show that the relations are the same on both sides. 
        
        By point 1, there is a regular sequence of invariants in $K[(\Delta_1)^{\times mn}]$:
        \[
        (t, t(x_{1,1} + x_{2,1} + \cdots + x_{mn,1}), t(x_{1,1}x_{2,1} + x_{1,1}x_{3,1} + \cdots x_{mn - 1,1}x_{mn,1}), \ldots, tx_{1,1}x_{2,1} \cdots x_{mn,1}).
        \]
        
        Since $S_{(m^n)} \subseteq S_{mn}$, this regular sequence is already in $K[(\Delta_1)^{\times mn}]^{S_{(m^n)}}$. Explicitly, for $i = 1, \dots, mn$, the elements of this regular sequence can be written as
        \[
        e_{1,mn,1}^i = \sum_{\lambda \in C_{m, n}(i)}x(\lambda),
        \]
        where $C_{m, n}(i)$ is the set of weak compositions of $i$ whose Young diagram is contained in the rectangle $[0, m] \times [0, n]$. Then, for $i = 1,\dots, mn$,
        \[
        \varphi(e_{1,mn,1}^i) = e_{m, n, 1}^i.
        \]
        The claim follows.
        
        \item This follows from \cite{braun-olsen-statistics}, proof of Theorem 4.10.
        \item These cases were verified using Macaulay2.
    \end{enumerate}
\end{proof}

We see that the inclusion $I_{m, n, k} \subseteq K[(\Delta_m)^{\times n}]_+^{Z_k \wr S_n}$ is proper when $m \geq 2$, as the following example shows.

\begin{example}\label{example:too-many-invariants}
    Let $m = n = 2$. By Example \ref{example:product-of-unit-simplices},
    \[
    K[(\Delta_2)^{\times 2}] = K\begin{bmatrix}
    tx_{2,2} & tx_{1,1}x_{2,2} & tx_{1, 2}x_{2,2}\\
    tx_{2,1} & tx_{1,1}x_{2,1} & tx_{1, 2}x_{2,1} \\
    t & tx_{1,1} & tx_{1, 2} & \\
    \end{bmatrix} \subseteq K[t, x_{1, 1}, \ldots, x_{2, 2}].
    \]
    Note that $S_2$ acts on the generators of $K[(\Delta_2)^{\times 2}]$ by reflecting along the diagonal going from southwest to northeast.
    
    The regular sequence for $k = 1$ is 
    \[
    (t, \hphantom{X} tx_{1, 1} + tx_{2, 1}, \hphantom{X} tx_{1, 2} + tx_{1, 1}x_{2, 1} + tx_{2, 2}, \hphantom{X} tx_{1, 2}x_{2, 1} + tx_{1, 1}x_{2, 2}, \hphantom{X} tx_{1, 2}x_{2, 2}).
    \]
    We see that the monomial $tx_{1,1}x_{2,1} \in K[(\Delta_2)^{\times 2}]^{S_2}_+$ but $tx_{1,1}x_{2,1} \not\in I_{2, 2, 1}$.
\end{example}

Even though $K[(\Delta_m)^{\times n}]$ may have too many invariants, studying $R_{m, n, k}$ is still very fruitful. Indeed, if $R$ is a Cohen-Macaulay (bi)graded $K$-algebra, and $e_0, e_1, \ldots, e_N$ is a homogeneous regular sequence of maximal length, it is a well-known fact that
\[
R \cong R/(e_0, e_1, \ldots, e_N) \otimes K[e_0, e_1, \ldots, e_N] 
\]
as (bi)graded $K$-vector spaces. In particular, if Conjecture \ref{conjecture:regular-sequence} holds,
\[
K[(\Delta_m)^{\times n}] \cong R_{m, n, k} \otimes K[e^0, e_{m, n, k}^1, e_{m, n, k}^2, \dots, e_{m, n, k}^{mn}]
\]
as (bi)graded $K$-vector spaces. This can be seen as an analogue of Chevalley's theorem for reflection groups. Since $I_{m, n, k}$ is invariant and homogeneous, this is also an isomorphism of bigraded $Z_k \wr S_n$-modules. The Frobenius characteristic then satisfies
\begin{equation}\label{align:frobenius-characteristic-of-cohen-macaulay-algebra}
    \Frob(K[(\Delta_m)^{\times n}]) =\frac{\Frob(R_{m, n, k})}{(1 - t) \prod_{i = 1}^{mn} (1 - t^kq^{ki})}.
\end{equation}

We will use this to give a geometric interpretation to equivariant Euler-Mahonian identities. This discussion, together with Corollary \ref{corollary:frobenius-characteristic}, also imply the following:

\begin{proposition}\label{proposition:equivariant-euler-mahonian-non-explicit}
    Let $P = \Delta_m \subseteq \RR^m$ and $\bm{a}' = (1, 2, \ldots, m)$. Then:
\begin{enumerate}
    \item If $k = 1$ and $m, n \in \NN_+$, there exists a polynomial $h_{S_n, \lambda}(t, q) \in \NN[t, q]$ such that the Frobenius characteristic of $K[(\Delta_m)^{\times n}]$ is given by
    \[
        \sum_{i \geq 0} t^i \sum_{\lambda \vdash n} s_\lambda \left [ \qbinom{i + m}{m}_q   \right ] s_\lambda[X] = \frac{\sum_{\lambda \vdash n} h_{S_n, \lambda}(t, q) s_\lambda[X]}{\prod_{j = 0}^{mn} (1 - tq^j)}.
    \]
\item In general, if Conjecture~\ref{conjecture:regular-sequence} is true, there exists a polynomial $h_{Z_k \wr S_n, \vec{\lambda}}(t, q) \in \mathbb{N}[t, q]$ such that
    \[
        \Frob\left(K[(\Delta_m)^{\times n}]\right)= \frac{\sum_{\vec{\lambda} \vdash_k n} h_{Z_k \wr S_n, \vec{\lambda}}(t, q) s_{\vec{\lambda}}(X)}{(1 - t)\prod_{j = 1}^{mn} (1 - t^kq^{kj})}.
    \]
\end{enumerate}
\end{proposition}

If $m = 1$, we can use our theory of color rules to give an explicit combinatorial formula for $h_{Z_k \wr S_n, \vec{\lambda}}(t, q)$. The case $m = 1, k = 1$ will already be very interesting. To proceed, we must define the statistics involved in this formula.

\begin{definition}[Wreath statistics for tableaux]\label{definition:wreath_statistics_tableaux}
    Fix a sequence of partitions $\vec{\gamma} \vdash_k n$ (drawn corner to corner as in Figure \ref{fig:partitions}). For $T\in \SYT(\vec{\gamma})$, let $W(T)$ be the tableau one gets from $T$ by the following procedure: replace the label $i$ in $T$ by $a_i$, where
    \begin{enumerate}
        \item if the label $1$ is in $\gamma^{r}$, we set $a_1 = r$. Then,
        \item suppose we have just replaced the label $i$ in $T$, located in $\gamma^r$. Suppose $i+1$ is in $\gamma^s$. 
        \begin{enumerate}
        \item If $i+1$ is strictly to the right of $i$, then we set $a_{i+1} = a_i + s-r.$
        \item If $i+1$ is weakly to the left of $i$, then set $a_{i+1} = a_i + k + s-r.$
        \end{enumerate}
    \end{enumerate}
    Now define ${\wcomaj}(T) = |W(T)|$ and ${\wdes}(T) = \max(W(T))$ is the maximal label in $W(T)$. See Figure \ref{fig:wreathstats} for an example.
\end{definition}
\begin{figure}[h]
    \centering
    \begin{tikzpicture}[scale = 1/2, baseline = 10ex]
 \draw [color = black!60, thick] (0,4) -- (1,4);
\draw [color = black!60, thick] (1,3) -- (1,4);
\draw [color = black!60, thick] (2,3) -- (2,4);
\draw [color = black!60, thick] (4,1) -- (6,1);
\draw [color = black!60, thick] (5,0) -- (5,2);
\draw [black, thick] (0,3) -- (3,3)  -- (3,2)  -- (4,2) -- (4, 0) -- (6,0) -- (6,2) --(4,2) -- (4,3) -- (3,3) -- (3,4) -- (1,4) -- (1,5) -- (0,5) -- cycle;
\node at (3.5,2.5) {$\emptyset$};
\node at (-2.5,2.5) {$T=$};
\node at (.5,3.5) {$1$};
\node at (1.5,3.5) {$2$};
\node at (4.5,.5) {$3$};
\node at (4.5,1.5) {$4$};
\node at (.5,4.5) {$5$};
\node at (2.5,3.5) {$6$};
\node at (5.5,.5) {$7$};
\node at (5.5,1.5) {$8$};
\node at (7.5,2.5) {\Large$\mapsto$};
\end{tikzpicture}
\begin{tikzpicture}[scale = 1/2, baseline = 10ex]
 \draw [color = black!60, thick] (0,4) -- (1,4);
\draw [color = black!60, thick] (1,3) -- (1,4);
\draw [color = black!60, thick] (2,3) -- (2,4);
\draw [color = black!60, thick] (4,1) -- (6,1);
\draw [color = black!60, thick] (5,0) -- (5,2);
\draw [black, thick] (0,3) -- (3,3)  -- (3,2)  -- (4,2) -- (4, 0) -- (6,0) -- (6,2) --(4,2) -- (4,3) -- (3,3) -- (3,4) -- (1,4) -- (1,5) -- (0,5) -- cycle;
\node at (3.5,2.5) {$\emptyset$};
\node at (-2.5,2.5) {$W(T)=$};
\node at (.5,3.5) {$0$};
\node at (1.5,3.5) {$0$};
\node at (4.5,.5) {$2$};
\node at (4.5,1.5) {$5$};
\node at (.5,4.5) {$6$};
\node at (2.5,3.5) {$6$};
\node at (5.5,.5) {$8$};
\node at (5.5,1.5) {$11$};
\end{tikzpicture}
    \caption{The standard tableau $T$ and image $W(T)$ give $\wcomaj(T) = |W(T)| = 0+0+2+5+6+6+8+11$ and $\wdes(T) = \max(W(T)) = 11.$}
    \label{fig:wreathstats}
\end{figure}

This notion of wreath comajor index and number of wreath descents can also be defined for elements in $Z_k \wr S_n$.

\begin{definition}[Wreath statistics for colored permutations]\label{definition:wreath_statistics_permutations}
The {\it wreath descent set} of an element $\sigma = u_{a_1} \sigma_1 \cdots u_{a_n} \sigma_n$ is an ordered list of sets
$\wDes(\sigma) = (D^1,\dots, D^{k})$
constructed in the following way:
\begin{enumerate}
    \item If $a_{i+1} - a_i = r \mod k$, with $r\in \{1,\dots, k-1\}$, then we say $i$ is an $r$-descent and $i \in D^r$.
    \item If $a_i = a_{i+1}$ and $\sigma_i > \sigma_{i+1}$, then we say that $i$ is a $k$-descent, and $i \in D^k$.
\end{enumerate}
The {\it wreath comajor index} and the number of {\it wreath descents} is given by
\begin{align*}
\wcomaj(\sigma) = 
\sum_{r=0}^{k} \sum_{i\in D^r} r \cdot (n-i) && \text{ and } &&
\wdes(\sigma) = \sum_{r=0}^{k} |D^r| \cdot r.
\end{align*}
\end{definition}

\begin{remark}
    One can define an RSK algorithm
    \[
    \sigma \in Z_k \wr S_n \mapsto 
    (P(\sigma),Q(\sigma))\in \cup_{\vec{\gamma}\vdash_k n} \SYT(\vec{\gamma})^{\times 2}
    \]
    by taking $\sigma = u_{a_1} \sigma_1 \cdots u_{a_n} \sigma_n$ and applying the usual bumping algorithm, only that the index $\sigma_i$ is inserted into $\gamma^{a_i}$. The recording tableau $Q(\sigma)$ then has the property that $\wcomaj(\sigma) = \wcomaj(Q(\sigma))$ and $\wdes(\sigma) = \wdes(Q(\sigma)).$
\end{remark}

\begin{theorem}\label{theorem:equivariant-euler-mahonian}
    Let $P = \Delta_1 = [0, 1] \subseteq \RR^1$ and $\bm{a}' = (1)$. Then, for all $k, n \in \NN_+$, the bigraded Frobenius characteristic of $K[(\Delta_1)^{\times n}]$ is given by
\begin{equation}\label{equation:equivariant-euler-mahonian-3}
   \sum_{r \geq 0} t^r  \sum_{\vec{\lambda} \vdash_k n} \prod_{i=0}^{k-1} \prod_{j = 1}^{\ell(\lambda^i)} [r + 1]_{q_{i, j}} \overline{P}_{\vec{\lambda}}[X]/Z_{\vec{\lambda}} = \frac{\sum_{\vec{\lambda} \vdash_k n}\sum_{T \in \SYT(\vec{\lambda})} t^{{\wdes}(T)} q^{{\wcomaj}(T)} s_{\vec{\lambda}}[X]}{(1 - t)\prod_{j = 1}^n (1 - t^kq^{kj})},
\end{equation}
where $q_{i, j} = u_i q^{\lambda(\sigma)_j^i}$ and $[n]_q = 1 + q + \cdots + q^{n - 1}$ is the $q$-integer. \\
        
In particular, for $k=1$, 
\begin{equation}\label{equation:equivariant-euler-mahonian-1}
            \sum_{i \geq 0} t^i \sum_{\lambda \vdash n} s_\lambda \left [ [i + 1]_q \right ] s_\lambda[X] = \frac{\sum_{\lambda \vdash n} \sum_{T \in \SYT(\lambda)} t^{\des(T)} q^{\comaj(T)} s_\lambda[X]}{\prod_{j = 0}^n (1 - tq^j)}.
        \end{equation}
These identities can be interpreted geometrically in view of Equation \ref{align:frobenius-characteristic-of-cohen-macaulay-algebra}.
\end{theorem}

\begin{proof}
We give the proof for the general case in Subsection~\ref{subsection:proof-of-theorem-5-X}. \\

However, the consequence for $k=1$ can be proven on its own with more classical methods (see Subsection \ref{subsection:euler-mahonian-identities} for a historical discussion). Indeed, note that
\[
L_{\Delta_1, i}^{(1)}(q) = 1+ q + \cdots + q^i = \frac{1 - q^{i + 1}}{1 - q} = [i + 1]_q.
\]
Then, the left-hand side of Equation~\eqref{equation:equivariant-euler-mahonian-1}
is $\Frob ((\Delta_1)^{\times n})$ in Corollary \ref{corollary:frobenius-characteristic}. To obtain the right hand side, use \cite{stanley-enumerative-2}, Proposition 7.9.12, and standard properties of the $q$-binomial coefficient. 
\end{proof}

\begin{remark}
    The wreath statistics in Theorem \ref{theorem:equivariant-euler-mahonian} are related to the negative and flag statistics in \cite{bagno-euler}, \cite{bagno-biagioli-colored}, \cite{braun-olsen-statistics}, since after projecting the equivariant formulas to the non-equivariant ones we see that their joint distribution must be the same.
\end{remark}

\subsection{Proof of Theorem~\ref{theorem:equivariant-euler-mahonian}}\label{subsection:proof-of-theorem-5-X}

By Proposition~\ref{proposition:polytope-character}, the left-hand side of Equation~\ref{equation:equivariant-euler-mahonian-3} gives
    \[
    \Frob(\chi^{K[(\Delta_1)^{\times n}]}) = 
    \frac{1}{n! k^n}
    \sum_{\sigma \in Z_k \wr S_n} \chi^{K[(\Delta_1)^{\times n}]}(\sigma)
\overline{P}_{\vec{\lambda}(\sigma)}[X].
    \]
Since $d \Delta_1 \cap \mathbb{Z} =\{0,\dots,d\}$, by Theorem~\ref{thm:projectivecharacter}, we know that
\[
    \frac{1}{n! k^n}
    \sum_{\sigma \in Z_k \wr S_n} \chi^{K[(\Delta_1)^{\times n}]}(\sigma)
\overline{P}_{\vec{\lambda}(\sigma)}[X]
=
\sum_{d \geq 0} t^d \sum_{T \in \SSYT_k(\vec{\gamma},\mathbb{N}_d)} \rho(T),
    \]
where $\SSYT_k(\vec{\gamma},\mathbb{N}_d)$ is the set of semistandard tableaux with entries no greater than $d$, and $i$ appears in $\gamma^j$ only if $i = j \mod k$. These tableaux have a $q$-weight given by the sum of elements in $T$,
$
\rho(T) = |T|.
$
Let
\[
\max T = \max\{ \beta_i: \text{ $\beta$ appears in $T$} \}.
\]

We first note that 
\[
\sum_{d \geq 0} t^d \sum_{T \in \SSYT_k(\vec{\gamma},\mathbb{N}_d)} 
q^{|T|} = \frac{1}{1-t} \sum_{T \in \SSYT_k(\vec{\gamma},\mathbb{N})} t^{\max T} q^{|T|}.
\]
Let 
$R_{n} = \sum_{T \in \SSYT_k(\vec{\gamma},\mathbb{N})} t^{\max T} q^{|T|}.$
For $i = n,n-1,\dots, 1$, we set
\begin{equation} \label{equation:Ri}
 R_{i-1} = (1-t^kq^{ki}) R_{i} .
\end{equation}
We will construct a sequence of sets of tableaux $U_i$ for which 
\[
R_i = \sum_{T \in U_i} t^{\max T} q^{|T|}.
\]
To this end, for each $i$, we define an injection
\[
\varphi_i: U_i \rightarrow U_i
\]
with the property that 
\begin{itemize}
    \item $|\varphi(T)| = ki+|T|,$ and 
    \item $\max \varphi(T) = k + \max T$,
\end{itemize}
meaning that Equation~\ref{equation:Ri} is satisfied if we set
\[
U_{i-1} = U_i \setminus  \varphi_i(U_i).
\]
In the end, we will have found that
\[
\prod_{i=1}^n (1-t^kq^{ki}) R_n = \sum_{T \in U_0} t^{\max T} q^{|T|}.
\]

For a given semistandard tableau $T$, we let $\stand(T)$ be the standardization of $T$ as described in Section~\ref{section:Preliminaries}.
Let $c_i=c_i(T)$ be the column in which the label $i$ is located in $\stand(T)$; and let $a_i = a_i(T)$ be the corresponding label in $T$.
In this notation, columns $1$ to $\gamma^1_1$ are columns of $\gamma^1$; columns $\gamma^1_1+1,\dots, \gamma^1_1+\gamma^2_1$ are columns of $\gamma^2$; and so on. We will then say that the column $c_i$ is in $\gamma^j$ if the label $i$ of $\stand(T)$ is in $\gamma^j$.
\\

Define $\varphi_i(T)$ to be the resulting tableau one gets by adding $k$ to the labels in $T$ corresponding to the labels $n,n-1,\dots, n-i+1$ of $\stand(T)$. 
Note that $\varphi(U_n)$ is the set of all tableaux in $U_n$ whose minimal label is at least $k$. This means
$U_{n-1}$ consists of all  tableaux in $U_n$ with the following property:
\begin{center}
if $c_1$ is in $\gamma^i$, then $a_i = i$. 
\end{center}
In general, the set $U_{n-j-1}$ is attained from $U_{n-j}$ by adding an additional condition: suppose $c_j$ is in $\gamma^{r}$ and $c_{j+1}$ is in $\gamma^{s}$.
\begin{itemize}
    \item If $c_j<c_{j+1}$ and $r=s$, then $a_{j+1} = a_j+{s-r}.$
    \item If $c_j \geq c_{j+1}$, then $a_{j+1} = a_j + k + s - r $.
\end{itemize}
The set $U_0$ satisfies these two conditions for each label $j$ in $\stand(T)$. Thus, each $S \in \SYT(\vec{\gamma})$ corresponds to a unique element $W(T) \in U_0$, where the map $W$ is defined in Definition~\ref{definition:wreath_statistics_tableaux}.
\\

We have found that, for a given $n,k$ and $\vec{\gamma} \vdash_k n$, 
\[
(1-t)\left(\prod_{j=1}^n 1-t^k q^{kj} \right) \sum_{d \geq 0} t^d \sum_{T \in \SSYT_{k}(\vec{\gamma},\mathbb{N}_d)} q^{|T|} t^{\max(T)}
=
\sum_{T \in \SYT(\vec{\gamma})} q^{\wcomaj(T)} t^{\wdes(T)}.
\]

This proves the result.

\section{Proof of the main result}\label{section:main_proof}
Here, we prove the main result regarding class functions computed by color rules, Theorem~\ref{theorem:mainresult}.
\subsection{The combinatorial objects}
We start with a class function $\chi$ of $Z_k \wr S_n$ computed by a color rule $F=[f_1^{m_1},f_2^{m_2},\dots]$ with value and weight functions $p$ and $\rho$, as in Definition~\ref{definition:color_rule}. Let $\gamma \vdash_k n$. The multiplicity of $\chi^{\vec{\gamma}}$ in $\chi$ is given by the inner product
\[
\langle \chi, \chi^{\vec{\gamma}} \rangle = \frac{1}{|Z_k \wr S_n|} \sum_{\sigma \in Z_k \wr S_n} \chi(\sigma) \overline{\chi}^{\vec{\gamma}}(\sigma).
\]

We will show that 
\[
\sum_{\sigma \in Z_k \wr S_n} \chi(\sigma) \overline{\chi}^{\vec{\gamma}}(\sigma) = n! k^n \sum_{T \in \SSYT_k(\vec{\gamma}, F)} \rho(T),
\]
as in Definition~\ref{definition:SSTYk}. We start by constructing a set of combinatorial objects $\mathcal{P}_{\vec{\gamma},F}$ and a weight function $w$ for which
\[
\sum_{P \in \mathcal{P}_{\vec{\gamma},F}} \weight(P) = \sum_{\sigma \in Z_k \wr S_n} 
\chi(\sigma) \overline{\chi}^{\vec{\gamma}}(\sigma).
\]

The set $\mathcal{P}_{\vec{\gamma},F}$ is created by the following procedure. One can follow Example~\ref{example:objects} to see how the steps create these objects:

\begin{enumerate}
\item Write $\sigma$ in cycle notation.
\item Over each cycle, choose a color $f_i$ and place it over each index in the cycle so that the resulting coloring satisfies the color rule.

\item Rearrange the cycles first in increasing order by color, so that cycles colored by $f_j$ appear to the right of cycles colored by $f_i$ when $i<j$.
Then, for cycles with the same color, rearrange in deacreasing cycle order (as described in Section~\ref{section:Preliminaries}). 

Let $(c_1,\dots, c_l)$ be the cycles from {\emph{ right}} to {\emph{left}} with lengths given by 
$\alpha = (\alpha_1,\dots,\alpha_l)$ and C-types given by $a =(a_1,\dots,a_l)$; and let $(g_1,\dots, g_l)$ be the colors over these respective cycles.

\item Pick a rim hook tableau $T$ of shape $\vec{\gamma}$ and type $(\alpha,a).$

\item If $\zeta_1, \dots, \zeta_l$ are the rim hooks in $T$ in the order in which they were placed, write over $\zeta_j$ by placing the first index of $c_j$ and its root of unity in the first cell of $\zeta_j$; place the second index of $c_j$ and its root of unity in the second cell of $\zeta_j$; and so on.  Similarly, include the colors so that each cell in $\zeta_j$ also contains the color $g_j$. Let $P$ be the resulting tableau.

\item Suppose $u_{r_i} i$ appears in $\gamma^{b_i}$ with color $f_{s_i}$ over it. Let $d(T)$ be the total number of South steps in all of the rim hooks appearing in $T$. Define
\[
\weight(P) \coloneq (-1)^{d(T)} \prod_{i=1}^n (u_{r_i})^{p(f_{s_i})-b_i} \rho(f_{s_i}).
\]
\end{enumerate}
 We call $\mathcal{P}_{\vec{\gamma},F}$ the collection of rim hook tableau fillings with colors in $F$. By construction, we have the following.
 \begin{proposition} \label{proposition:combiantorialexpansion}
For any $\vec{\gamma} \vdash_k n$ and color rule $F$,
\[
\sum_{P \in \mathcal{P}_{\vec{\gamma},F}}\weight(P) =  \sum_{\sigma \in Z_k \wr S_n} 
\chi(\sigma) \overline{\chi}^{\vec{\gamma}}(\sigma).
\]
 \end{proposition}
 \begin{example}\label{example:objects}
As an example of this construction, for the case when $n = 8$, $k = 3$, and $\vec{\gamma} = ((2,1),(2,2),(1))$,
consider the element 
\begin{equation*}
(1)(2~ u_16~ 4) (u_2 3)  (u_15 )(u_2 7 ~ u_1 8),
\end{equation*}

We choose where to place the colors from the color rule $[f_1^6,f_2,f_3]$. One such possibility is 
\begin{center}
\begin{tikzpicture}[scale = 3/4]
\node at (.75,0) {$(1)$};
\node at (4.5,0) {$(u_2 3)$};
\node at (2.5,0) {$(2 ~ u_1 6~ 4)$};
\node at (6,0) {$(u_1 5)$};
\node at (8,0) {$(u_2 7 ~ u_1 8)$};
\node at (.75,.6) {$f_3$};
\node at (4.5,.6) {$f_1$};
\node at (2.5,.6) {$f_1$};
\node at (2.5+.7,.6) {$f_1$};
\node at (2.5-.7,.6) {$f_1$};
\node at (6,.6) {$f_2$};
\node at (8.5,.6) {$f_1$};
\node at (8-.5,.6) {$f_1$};
\end{tikzpicture}.
\end{center}

We then rearrange the cycles so colors are increasing left to right, and the minimal indices in cycles with the same color are decreasing from left to right:
\begin{center}
\begin{tikzpicture}[scale = 3/4]
\node at (2.5,0) {$(u_2 3)$};
\node at (4.5,0) {$(2 ~ u_1 6~ 4)$};
\node at (.5,0) {$(u_2 7 ~ u_1 8)$};
\node at (6.5,0) {$(u_15)$};
\node at (8,0) {$(1)$};
\node at (2.5,.6) {$f_1$};
\node at (8,.6) {$f_3$};
\node at (4.5,.6) {$f_1$};
\node at (4.5+.7,.6) {$f_1$};
\node at (4.5-.7,.6) {$f_1$};
\node at (6.5,.6) {$f_2$};
\node at (1,.6) {$f_1$};
\node at (0,.6) {$f_1$};
\end{tikzpicture}.
\end{center}

We now pick a rim hook tableau of shape $((2,1),(2,2),(1))$ whose rim hooks have lengths given by these cycles read right to left: $(1,1,3,1,2)$. Say we pick

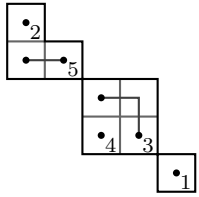
\begin{align*}
\begin{tikzpicture}[scale = 1/2, baseline = 10ex]
\draw [color = black!60, thick] (0,5) -- (1,5);
\draw [color = black!60, thick] (1,4) -- (1,5);
\draw [color = black!60, thick] (2,3) -- (4,3);
\draw [color = black!60, thick] (3,2) -- (3,4);
\draw[color=black!75, thick] (.5, 4.5) -- (1.5,4.5);
\fill[color = black, thick] (.5,4.5) circle (.5ex);
\fill[color = black, thick] (1.5,4.5) circle (.5ex);
\fill[color = black, thick] (.5,5.5) circle (.5ex);
\draw[color=black!75, thick] (2.5, 3.5) -- (3.5, 3.5) -- (3.5,2.5);
\fill[color = black, thick] (2.5,3.5) circle (.5ex);
\fill[color = black, thick] (3.5,2.5) circle (.5ex);
\fill[color = black, thick] (4.5,1.5) circle (.5ex);
\fill[color = black, thick] (2.5,2.5) circle (.5ex);
\node at (.75,5.25) {\scriptsize{$2$}};
\node at (3.75,2.25) {\scriptsize{$3$}};
\node at (1.75,4.25) {\scriptsize{$5$}};
\node at (4.75, 1.25) {\scriptsize{$1$}};
\node at (2.75, 2.25) {\scriptsize{$4$}};
\draw [black, thick] (0,4) -- (2,4)  -- (2,2) -- (4,2) -- (4,1) -- (5,1) -- (5,2) --(4,2) -- (4,4) -- (2,4) -- (2,5) --(1,5) -- (1,6) -- (0,6) -- cycle;
\end{tikzpicture}
\end{align*}

Finish by filling in the shape by the entries in the cycles to obtain

\begin{align*}
\begin{tikzpicture}[scale = 1, baseline = 10ex]
\draw [color = black!60, thick] (0,5) -- (1,5);
\draw [color = black!60, thick] (1,4) -- (1,5);
\draw [color = black!60, thick] (2,3) -- (4,3);
\draw [color = black!60, thick] (3,2) -- (3,4);
\draw[color=black!75, thick] (.5, 4.5-.15) -- (1.5,4.5-.15);
\fill[color = black, thick] (.5,4.5-.15) circle (.5ex);
\fill[color = black, thick] (1.5,4.5-.15) circle (.5ex);
\fill[color = black, thick] (.5,5.5-.15) circle (.5ex);
\draw[color=black!75, thick] (2.5, 3.5-.15) -- (3.5, 3.5-.15) -- (3.5,2.5-.15);
\fill[color = black, thick] (2.5,3.5-.15) circle (.5ex);
\fill[color = black, thick] (3.5,2.5-.15) circle (.5ex);
\fill[color = black, thick] (4.5,1.5-.15) circle (.5ex);
\fill[color = black, thick] (2.5,2.5-.15) circle (.5ex);
\node [draw,fill=white,inner sep=1pt] (A) at (.5,5.35) {$u_1 5$};
\node [scale = .8] (P) at (.5-.15,5.75) {$f_2$};
\node [draw,fill=white,inner sep=1pt] (B) at (.5,4.35) {$u_2 7$};
\node [scale = .8] (O) at (.5-.15,4.75) {$f_1$};
\node [draw,fill=white,inner sep=1pt] (C) at (1.5,4.35) {$u_1 8$};
\node [scale = .8] (N) at (1.5-.15,4.75) {$f_1$};
\node [draw,fill=white,inner sep=1pt] (D) at (2.5,3.35) {$2$};
\node [scale = .8] (M) at (2.5-.15,3.75) {$f_1$};
\node [draw,fill=white,inner sep=1pt] (E) at (3.5,3.35) {$u_1 6$};
\node [scale = .8] (L) at (3.5-.15,3.75) {$f_1$};
\node [draw,fill=white,inner sep=1pt] (F) at (3.5,2.35) {$4$};
\node [scale = .8] (K) at (3.5-.15,2.75) {$f_1$};
\node [draw,fill=white,inner sep=1pt] (G) at (4.5,1.35) {$1$};
\node [scale = .8] (H) at (4.5-.15,1.75) {$f_3$};
\node [draw,fill=white,inner sep=1pt] (I) at (2.5,2.35) {$u_2 3$};
\node [scale = .8] (J) at (2.5-.15,2.75) {$f_1$};
\node [scale = 1] (P) at (-1.5,3.25) {$P = $};
\draw [black, thick] (0,4) -- (2,4)  -- (2,2) -- (4,2) -- (4,1) -- (5,1) -- (5,2) --(4,2) -- (4,4) -- (2,4) -- (2,5) --(1,5) -- (1,6) -- (0,6) -- cycle;
\end{tikzpicture}
\end{align*}

We calculate the weights as follows:
\begin{align*}
\weight(P) & = (-1)^1 ( u_0^{p(f_3)-2} \rho(f_3)) ( u_0^{p(f_1)-1} \rho(f_1)) ( u_2^{p(f_1)-1} \rho(f_1)) \\
& ~~~~~\times ( u_0^{p(f_1)-1} \rho(f_1)) ( u_1^{p(f_2)-0} \rho(f_2)) ( u_1^{p(f_1)-1} \rho(f_1)) ( u_2^{p(f_1)-0} \rho(f_1)) ( u_1^{p(f_1)-0} \rho(f_1)).
\end{align*}
To make it easier to follow, we first start with the sign given by the number of South steps in the rim hooks. There is only one, giving $(-1)^1$. Then, for each $i=1,\dots, 8$ (in this order), we write $u_{r_i}^{p(f_{s_i}-b_i)} \rho(f_{s_i})$, where $u_{r_i} i$ appears in $\gamma^{b_i}$ with color $f_{s_i}$. For instance, since $u_1 6$ appears in $\gamma^1$ with color $f_1$, we get $u_1^{p(f_1)-1}\rho(f_1)$.
\end{example}

\subsection{A weight-preserving, sign-reversing involution and a map of order k}

To evaluate the first sum in Proposition \ref{proposition:combiantorialexpansion}, we must eliminate all non-integral and non-positive weights. To do this, we must first use a weight-preserving, sign-reversing involution to eliminate a certain possibility for our objects' fillings. We will then employ another fact of the cyclic group to eliminate a broader range of possibilities. To facilitate our discussion, we will say that a cell is disconnected from another cell if a rim hook between the two cells is severed, and we will say that a cell is connected to another cell if a rim hook from one cell is extended to the other.

We first define a sign-reversing involution $\psi$ on $\mathcal{P}_{\vec{\gamma},F}$.
Examples of the operation $\psi$ can be found in Figures~\ref{figure:involution_example_1} and \ref{figure:involution_example_2}.
For a given $P \in \mathcal{P}_{\vec{\gamma},F}$, define $\psi(P)$
by the following process: 
\begin{enumerate}
\item Scan along the southernmost row, from left to right, until one finds the first cell $c$ such that one of two things hold:
\begin{enumerate}
\item The cell above $c$ is in the same rim hook as $c$ or
\item The cell above $c$ contains the end of a rim hook with the same color in $c$.
\end{enumerate}
\item If no such cell is found in this row, move inductively to the next row up. 
\item If no such cell is found in any row, leave the object fixed.
\end{enumerate}

Suppose we now have such a cell $c$, then for the given cases, we do the following:
\begin{enumerate}[\text{Case }(a):]
\item If $c$ is in the same rim hook as the cell above, then
\begin{enumerate}[1.]
\item Disconnect $c$ from the cell above it.
\item Suppose there is a rim hook $\zeta$ which ends one cell west of $c$ and has the same color as $c$.  If the 
index in $c$ is larger than the smallest index in $\zeta$, then connect $c$ with $\zeta$.
\item Let $\xi$ be the rim hook which now contains $c$.  Read the indices in $\xi$ from right to 
left, looking for the first index $i$ in $\xi$ which is smaller than every index in $\xi$ on its left.  
If such  an index is found, disconnect the index from $\xi$ from the left.  
Iterate this procedure with the remaining portion of $\xi$.  This step ensures that each rim hook
begins with its smallest index.
\end{enumerate}
\end{enumerate}

\begin{enumerate}[\text{Case }(b):]
\item If $c$ is one cell south the end of a rim hook containing the same color, then
\begin{enumerate}[1.]
\item If $c$ is in the same rim hook as the cell to its west, disconnect these two cells.
\item Connect $c$ with the cell above $c$.  
\item Let $\zeta$ be the rim hook which now contains $c$.  Suppose there is a rim hook $\xi$ which 
begins one cell east of the end of $\zeta$.  If the smallest index in $\zeta$ is smaller than the 
smallest index in $\xi$, connect $\zeta$ and $\xi$.   Iterate this process until either there is no
rim hook $\xi$ that begins one cell east of the rim hook containing $c$ or the smallest index 
in $\zeta$ is larger than the smallest index in $\xi$.   
\end{enumerate}
\end{enumerate}

As an example of this operation, consider the tableau filling in Example~\ref{example:objects}. We see in Figure~\ref{figure:involution_example_1} that the first cell $c$ that $\psi$ would locate is the cell containing $4$. We then disconnect from above, and connect to the cell on the left, since $3<4$. Figure~\ref{figure:involution_example_2} gives a more drastic example.

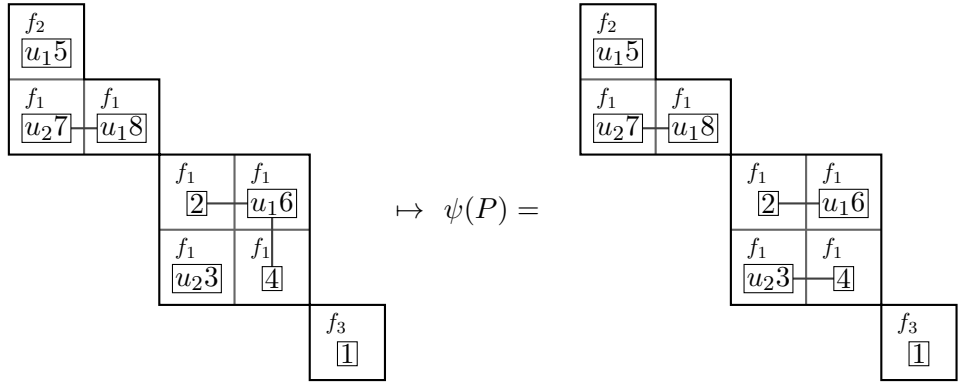
\begin{figure}
\begin{align*}
\begin{tikzpicture}[scale = 1, baseline = 10ex]
\draw [color = black!60, thick] (0,5) -- (1,5);
\draw [color = black!60, thick] (1,4) -- (1,5);
\draw [color = black!60, thick] (2,3) -- (4,3);
\draw [color = black!60, thick] (3,2) -- (3,4);
\draw[color=black!75, thick] (.5, 4.5-.15) -- (1.5,4.5-.15);
\fill[color = black, thick] (.5,4.5-.15) circle (.5ex);
\fill[color = black, thick] (1.5,4.5-.15) circle (.5ex);
\fill[color = black, thick] (.5,5.5-.15) circle (.5ex);
\draw[color=black!75, thick] (2.5, 3.5-.15) -- (3.5, 3.5-.15) -- (3.5,2.5-.15);
\fill[color = black, thick] (2.5,3.5-.15) circle (.5ex);
\fill[color = black, thick] (3.5,2.5-.15) circle (.5ex);
\fill[color = black, thick] (4.5,1.5-.15) circle (.5ex);
\fill[color = black, thick] (2.5,2.5-.15) circle (.5ex);
\node [draw,fill=white,inner sep=1pt] (A) at (.5,5.35) {$u_1 5$};
\node [scale = .8] (P) at (.5-.15,5.75) {$f_2$};
\node [draw,fill=white,inner sep=1pt] (B) at (.5,4.35) {$u_2 7$};
\node [scale = .8] (O) at (.5-.15,4.75) {$f_1$};
\node [draw,fill=white,inner sep=1pt] (C) at (1.5,4.35) {$u_1 8$};
\node [scale = .8] (N) at (1.5-.15,4.75) {$f_1$};
\node [draw,fill=white,inner sep=1pt] (D) at (2.5,3.35) {$2$};
\node [scale = .8] (M) at (2.5-.15,3.75) {$f_1$};
\node [draw,fill=white,inner sep=1pt] (E) at (3.5,3.35) {$u_1 6$};
\node [scale = .8] (L) at (3.5-.15,3.75) {$f_1$};
\node [draw,fill=white,inner sep=1pt] (F) at (3.5,2.35) {$4$};
\node [scale = .8] (K) at (3.5-.15,2.75) {$f_1$};
\node [draw,fill=white,inner sep=1pt] (G) at (4.5,1.35) {$1$};
\node [scale = .8] (H) at (4.5-.15,1.75) {$f_3$};
\node [draw,fill=white,inner sep=1pt] (I) at (2.5,2.35) {$u_2 3$};
\node [scale = .8] (J) at (2.5-.15,2.75) {$f_1$};
\draw [black, thick] (0,4) -- (2,4)  -- (2,2) -- (4,2) -- (4,1) -- (5,1) -- (5,2) --(4,2) -- (4,4) -- (2,4) -- (2,5) --(1,5) -- (1,6) -- (0,6) -- cycle;
\end{tikzpicture}
\begin{tikzpicture}[scale = 1, baseline = 10ex]
\node [scale = 1] (P) at (-1.5,3.25) {$\mapsto ~~ \psi(P) = $};
\draw [color = black!60, thick] (0,5) -- (1,5);
\draw [color = black!60, thick] (1,4) -- (1,5);
\draw [color = black!60, thick] (2,3) -- (4,3);
\draw [color = black!60, thick] (3,2) -- (3,4);
\draw[color=black!75, thick] (.5, 4.5-.15) -- (1.5,4.5-.15);
\fill[color = black, thick] (.5,4.5-.15) circle (.5ex);
\fill[color = black, thick] (1.5,4.5-.15) circle (.5ex);
\fill[color = black, thick] (.5,5.5-.15) circle (.5ex);
\draw[color=black!75, thick] (2.5, 3.5-.15) -- (3.5, 3.5-.15) ;
\draw[color=black!75, thick] (2.5, 2.5-.15) -- (3.5,2.5-.15);
\fill[color = black, thick] (2.5,3.5-.15) circle (.5ex);
\fill[color = black, thick] (3.5,2.5-.15) circle (.5ex);
\fill[color = black, thick] (4.5,1.5-.15) circle (.5ex);
\fill[color = black, thick] (2.5,2.5-.15) circle (.5ex);A
\node [draw,fill=white,inner sep=1pt] (A) at (.5,5.35) {$u_1 5$};
\node [scale = .8] (P) at (.5-.15,5.75) {$f_2$};
\node [draw,fill=white,inner sep=1pt] (B) at (.5,4.35) {$u_2 7$};
\node [scale = .8] (O) at (.5-.15,4.75) {$f_1$};
\node [draw,fill=white,inner sep=1pt] (C) at (1.5,4.35) {$u_1 8$};
\node [scale = .8] (N) at (1.5-.15,4.75) {$f_1$};
\node [draw,fill=white,inner sep=1pt] (D) at (2.5,3.35) {$2$};
\node [scale = .8] (M) at (2.5-.15,3.75) {$f_1$};
\node [draw,fill=white,inner sep=1pt] (E) at (3.5,3.35) {$u_1 6$};
\node [scale = .8] (L) at (3.5-.15,3.75) {$f_1$};
\node [draw,fill=white,inner sep=1pt] (F) at (3.5,2.35) {$4$};
\node [scale = .8] (K) at (3.5-.15,2.75) {$f_1$};
\node [draw,fill=white,inner sep=1pt] (G) at (4.5,1.35) {$1$};
\node [scale = .8] (H) at (4.5-.15,1.75) {$f_3$};
\node [draw,fill=white,inner sep=1pt] (I) at (2.5,2.35) {$u_2 3$};
\node [scale = .8] (J) at (2.5-.15,2.75) {$f_1$};
\draw [black, thick] (0,4) -- (2,4)  -- (2,2) -- (4,2) -- (4,1) -- (5,1) -- (5,2) --(4,2) -- (4,4) -- (2,4) -- (2,5) --(1,5) -- (1,6) -- (0,6) -- cycle;
\end{tikzpicture}
\end{align*}
\caption{The involution $\psi$ on the object in Example~\ref{example:objects}.}
\label{figure:involution_example_1}
\end{figure}

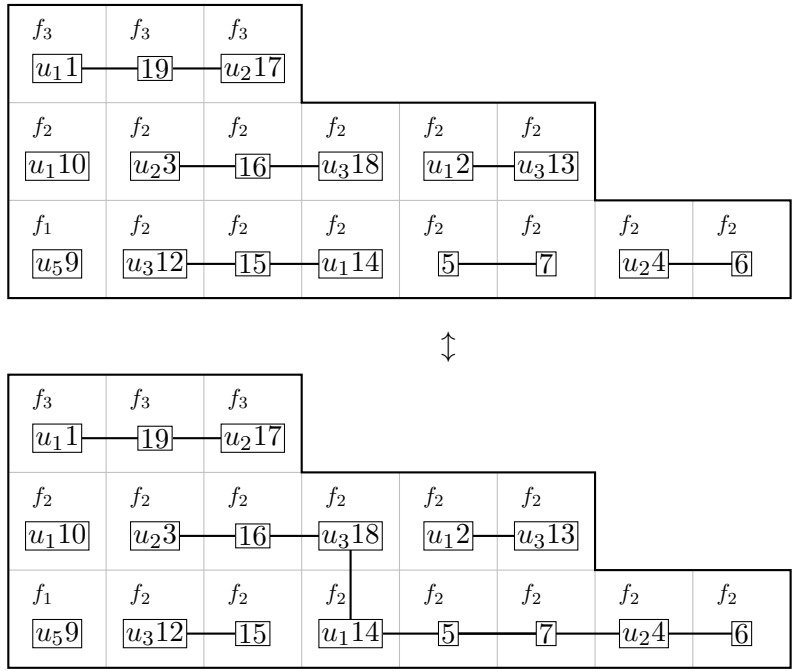
\begin{figure}
\begin{center}
\begin{tikzpicture}[scale = 1.3]
\draw [color = black!25] (1,0) -- (1,3);
\draw [color = black!25] (2,0) -- (2,3);
\draw [color = black!25] (3,0) -- (3,2);
\draw [color = black!25] (4,0) -- (4,2);
\draw [color = black!25] (5,0) -- (5,2);
\draw [color = black!25] (6,0) -- (6,1);
\draw [color = black!25] (7,0) -- (7,1);
\draw [color = black!25] (0,1) -- (6,1);
\draw [color = black!25] (0,2) -- (3,2);

\draw [thick] (0,0) -- (8,0) -- (8,1) -- (6,1) -- (6,2) -- (3,2) -- (3,3) -- (0,3) -- cycle;

\draw [thick] (1.5,1.5-.15) -- (3.5,1.5-.15); 
\draw [thick] (1.5,.5-.15) -- (3.5,.5-.15); 
\draw [thick] (4.5,.5-.15) -- (5.5,.5-.15); 
\draw [thick] (6.5,.5-.15) -- (7.5,.5-.15); 
\draw [thick] (.5,2.5-.15) -- (2.5,2.5-.15); 
\draw [thick] (4.5,1.5-.15) -- (5.5,1.5-.15); 
A
\node [draw,fill=white,inner sep=1pt] at (.5,.5-.15) {$u_5 9$};
\node [scale = .8] at (.5-.15,.75) {$f_1$};
\node [scale = .8] at (1.5-.15,.75) {$f_2$};
\node [scale = .8] at (2.5-.15,.75) {$f_2$};
\node [scale = .8] at (3.5-.15,.75) {$f_2$};
\node [scale = .8] at (4.5-.15,.75) {$f_2$};
\node [scale = .8] at (5.5-.15,.75) {$f_2$};
\node [scale = .8] at (6.5-.15,.75) {$f_2$};
\node [scale = .8] at (7.5-.15,.75) {$f_2$};
\node [scale = .8] at (.5-.15,1.75) {$f_2$};
\node [scale = .8] at (1.5-.15,1.75) {$f_2$};
\node [scale = .8] at (2.5-.15,1.75) {$f_2$};
\node [scale = .8] at (3.5-.15,1.75) {$f_2$};
\node [scale = .8] at (4.5-.15,1.75) {$f_2$};
\node [scale = .8] at (5.5-.15,1.75) {$f_2$};
\node [scale = .8] at (.5-.15,2.75) {$f_3$};
\node [scale = .8] at (1.5-.15,2.75) {$f_3$};
\node [scale = .8] at (2.5-.15,2.75) {$f_3$};
\node [draw,fill=white,inner sep=1pt] at (.5,1.5-.15) {$u_110$};
\node [draw,fill=white,inner sep=1pt] at (1.5,1.5-.15) {$u_23$};
\node [draw,fill=white,inner sep=1pt] at (2.5,1.5-.15) {$16$};
\node [draw,fill=white,inner sep=1pt] at (3.5,1.5-.15) {$u_318$};
\node [draw,fill=white,inner sep=1pt] at (.5,2.5-.15) {$u_11$};
\node [draw,fill=white,inner sep=1pt] at (1.5,2.5-.15) {$19$};
\node [draw,fill=white,inner sep=1pt] at (2.5,2.5-.15) {$u_217$};
\node [draw,fill=white,inner sep=1pt] at (4.5,1.5-.15) {$u_12$};
\node [draw,fill=white,inner sep=1pt] at (5.5,1.5-.15) {$u_313$};

\node [draw,fill=white,inner sep=1pt] at (1.5,.5-.15) {$u_312$};
\node [draw,fill=white,inner sep=1pt] at (2.5,.5-.15) {$15$};
\node [draw,fill=white,inner sep=1pt] at (3.5,.5-.15) {$u_114$};
\node [draw,fill=white,inner sep=1pt] at (4.5,.5-.15) {$5$};
\node [draw,fill=white,inner sep=1pt] at (5.5,.5-.15) {$7$};
\node [draw,fill=white,inner sep=1pt] at (6.5,.5-.15) {$u_24$};
\node [draw,fill=white,inner sep=1pt] at (7.5,.5-.15) {$6$};
\node at (4.5,-.5) {$\updownarrow$};
\end{tikzpicture} 
\\
\begin{tikzpicture}[scale = 1.3]
\draw [color = black!25] (1,0) -- (1,3);
\draw [color = black!25] (2,0) -- (2,3);
\draw [color = black!25] (3,0) -- (3,2);
\draw [color = black!25] (4,0) -- (4,2);
\draw [color = black!25] (5,0) -- (5,2);
\draw [color = black!25] (6,0) -- (6,1);
\draw [color = black!25] (7,0) -- (7,1);
\draw [color = black!25] (0,1) -- (6,1);
\draw [color = black!25] (0,2) -- (3,2);

\draw [thick] (0,0) -- (8,0) -- (8,1) -- (6,1) -- (6,2) -- (3,2) -- (3,3) -- (0,3) -- cycle;

\draw [thick] (1.5,1.5-.15) -- (3.5,1.5-.15) --(3.5,.5-.15) -- (6.5,.5-.15); 
\draw [thick] (1.5,.5-.15) -- (2.5,.5-.15); 
\draw [thick] (4.5,.5-.15) -- (5.5,.5-.15); 
\draw [thick] (6.5,.5-.15) -- (7.5,.5-.15); 
\draw [thick] (.5,2.5-.15) -- (2.5,2.5-.15); 
\draw [thick] (4.5,1.5-.15) -- (5.5,1.5-.15); 

\node [draw,fill=white,inner sep=1pt] at (.5,.5-.15) {$u_5 9$};
\node [scale = .8] at (.5-.15,.75) {$f_1$};
\node [scale = .8] at (1.5-.15,.75) {$f_2$};
\node [scale = .8] at (2.5-.15,.75) {$f_2$};
\node [scale = .8] at (3.5-.15,.75) {$f_2$};
\node [scale = .8] at (4.5-.15,.75) {$f_2$};
\node [scale = .8] at (5.5-.15,.75) {$f_2$};
\node [scale = .8] at (6.5-.15,.75) {$f_2$};
\node [scale = .8] at (7.5-.15,.75) {$f_2$};
\node [scale = .8] at (.5-.15,1.75) {$f_2$};
\node [scale = .8] at (1.5-.15,1.75) {$f_2$};
\node [scale = .8] at (2.5-.15,1.75) {$f_2$};
\node [scale = .8] at (3.5-.15,1.75) {$f_2$};
\node [scale = .8] at (4.5-.15,1.75) {$f_2$};
\node [scale = .8] at (5.5-.15,1.75) {$f_2$};
\node [scale = .8] at (.5-.15,2.75) {$f_3$};
\node [scale = .8] at (1.5-.15,2.75) {$f_3$};
\node [scale = .8] at (2.5-.15,2.75) {$f_3$};
\node [draw,fill=white,inner sep=1pt] at (.5,1.5-.15) {$u_110$};
\node [draw,fill=white,inner sep=1pt] at (1.5,1.5-.15) {$u_23$};
\node [draw,fill=white,inner sep=1pt] at (2.5,1.5-.15) {$16$};
\node [draw,fill=white,inner sep=1pt] at (3.5,1.5-.15) {$u_318$};
\node [draw,fill=white,inner sep=1pt] at (.5,2.5-.15) {$u_11$};
\node [draw,fill=white,inner sep=1pt] at (1.5,2.5-.15) {$19$};
\node [draw,fill=white,inner sep=1pt] at (2.5,2.5-.15) {$u_217$};
\node [draw,fill=white,inner sep=1pt] at (4.5,1.5-.15) {$u_12$};
\node [draw,fill=white,inner sep=1pt] at (5.5,1.5-.15) {$u_313$};

\node [draw,fill=white,inner sep=1pt] at (1.5,.5-.15) {$u_312$};
\node [draw,fill=white,inner sep=1pt] at (2.5,.5-.15) {$15$};
\node [draw,fill=white,inner sep=1pt] at (3.5,.5-.15) {$u_114$};
\node [draw,fill=white,inner sep=1pt] at (4.5,.5-.15) {$5$};
\node [draw,fill=white,inner sep=1pt] at (5.5,.5-.15) {$7$};
\node [draw,fill=white,inner sep=1pt] at (6.5,.5-.15) {$u_24$};
\node [draw,fill=white,inner sep=1pt] at (7.5,.5-.15) {$6$};
\end{tikzpicture} 
\end{center}
\caption{A larger example of the involution $\psi$.}
\label{figure:involution_example_2}
\end{figure}
\begin{proposition}
The map $\psi: \mathcal{P}_{\vec{\gamma},F} \rightarrow \mathcal{P}_{\vec{\gamma},F} $ is well-defined. Furthermore, $\psi$ is an involution, and if $P$ is not fixed, then $\weight(P) = - \weight(\psi(P)).$ 
\end{proposition}
\begin{proof}
The act of connecting and disconnecting rim hooks ensures that the cycles are inserted in the proper order, meaning that we indeed get a map from $\mathcal{P}_{\vec{\gamma},F}$ to itself. Furthermore, if $c$ is the cell located by the map $\psi$, then Case (a) is sent to Case (b) at $c$, and Case (b) is sent to Case (a) at $c$. This ensures that when we apply $\psi$ to $\psi(P)$, we locate the same cell $c$ to perform the required operation. So $\psi$ is an involution.

Secondly, no color or root of unity is changed, so the only difference between $\psi(P)$ and $P$ will be a sign given by the new underlying rim hook tableau.

Lastly, at most one South step in the rim hooks will be deleted or added by the operation. This is because there cannot be any South steps below $c$, since otherwise $\psi$ would have first located a cell below $c$. So Case (a), point 3 will never delete a South step. For the same reason, Case (b), point 3 never encounters the end of a rim hook to the right of $c$ that ends above a cell with the same color (otherwise $\psi$ would locate this cell first instead of $c$).
If $P$ is not fixed by $\psi$, precisely one South step in a rim hook is either added or removed, which reverses the sign.
\end{proof}
\begin{remark}\label{remark:P_prime_semistandard}
Let $\mathcal{P}_{\vec{\gamma},F}'$ be the fixed points of $\psi$. We first note that the underlying rim hook of any $P \in \mathcal{P}_{\vec{\gamma},F}'$ cannot have a South step, otherwise we would find a cell satisfying Case (a) in $\psi$. Secondly, we see that we cannot have two cells, one on top of the other, with the same color, since either we would have a South step in a rim hook, or there would be the end of a rim hook falling into Case (b) of $\psi$.
This means that the underlying colors form a semistandard tableau. 
\end{remark}
We will now use the following fact:
\begin{lemma} \label{lemma:sumofroots}
For any $s$ and $u_r$, we have
\begin{equation*}
u_r^s + (u_1 u_r)^s + \cdots + (u_1^{k-1} u_r)^s = k ~\delta_{(s = 0 \mod k)}
\end{equation*}
\end{lemma}
\begin{proof} 
If $s = 0 \mod k$, then we get $u_r^s =1$ for any $r$, and the sum reduces to $k$. Otherwise, we can rewrite the sum as 
\begin{equation*}
u_r^s (u_1^0+u_1^s+\cdots+u_1^{(k-1)s} ) = u_r^s \frac{1-(u_1^{s})^{k}}{1-u_1^s} = 0.
\end{equation*}
\end{proof}

\begin{proposition}
There is a map $\psi':\mathcal{P}_{\vec{\gamma},F}' \rightarrow \mathcal{P}_{\vec{\gamma},F}'$ of order $k$ such that if $P$ is not fixed, then 
$P$, $\psi'(P)$,$\dots$, ${\psi'}^{k-1}(P)$ are all distinct, and
$\weight(P)+\weight(\psi'(P))+\cdots + \weight({\psi'}^{k-1}(P)) = 0$.
\end{proposition}
\begin{proof}
Given $P \in \mathcal{P}_{\vec{\gamma},F}'$, scan from bottom to top, left to right, for the first cell $c$ in $\gamma^r$ containing a color $f$ such that $p(f) \neq r \mod k$. Suppose this cell contains $u_i j$. Let $\psi'(P)$ be the tableau filling obtained by replacing $u_i$ in $c$ by $u_1 u_i$. This clearly has order $k$ since $u_1^k = 1$.

The weight of the image of this map is given by $\weight(\psi'(P)) = u_1^{p(f)-r} \weight(P)$. So if $p(f)-r \neq 0 \mod k$, 
\begin{align*}
\weight(P)+\weight(\psi'(P))&+\cdots + \weight({\psi'}^{k-1}(P)) \\ & = \weight(P)( 
1+ u_1^{p(f)-r} + u_2^{p(f)-r}+\cdots + u_{k-1}^{p(f)-r}) = 0
\end{align*}
by Lemma \ref{lemma:sumofroots}.
\end{proof}
For a simple example of an orbit of $\psi'$, we have Figure~\ref{figure:orbit_psi_prime}. The sum of weights is displayed below each tableau filling.
\begin{figure}
\begin{align*}
\begin{tikzpicture}[scale = 1, baseline = 10ex]
\draw [color = black!60, thick] (0,1) -- (1,1);
\draw [color = black!60, thick] (1,0) -- (1,1);
\draw [thick] (.5,.5-.15) -- (1.5,.5-.15); 
\node [draw,fill=white,inner sep=1pt] (A) at (.5,.35) {$1$};
\node [draw,fill=white,inner sep=1pt] (B) at (1.5,.35) {$3$};
\node [draw,fill=white,inner sep=1pt] (C) at (.5,1.35) {$u_1 2$};
\node [scale = .8] (D) at (.5,.75) {$f_1$};
\node [scale = .8] (D) at (1.5,.75) {$f_1$};
\node [scale = .8] (D) at (.5,1.75) {$f_2$};
\draw [black, thick] (0,0) -- (2,0) -- (2,1) -- (1,1) -- (1,2) -- (0,2) -- cycle;
\end{tikzpicture}
&&
\begin{tikzpicture}[scale = 1, baseline = 10ex]
\draw [color = black!60, thick] (0,1) -- (1,1);
\draw [color = black!60, thick] (1,0) -- (1,1);
\draw [thick] (.5,.5-.15) -- (1.5,.5-.15); 
\node [draw,fill=white,inner sep=1pt] (A) at (.5,.35) {$u_1 1$};
\node [draw,fill=white,inner sep=1pt] (B) at (1.5,.35) {$3$};
\node [draw,fill=white,inner sep=1pt] (C) at (.5,1.35) {$u_1 2$};
\node [scale = .8] (D) at (.5,.75) {$f_1$};
\node [scale = .8] (D) at (1.5,.75) {$f_1$};
\node [scale = .8] (D) at (.5,1.75) {$f_2$};
\draw [black, thick] (0,0) -- (2,0) -- (2,1) -- (1,1) -- (1,2) -- (0,2) -- cycle;
\end{tikzpicture}
&&
\begin{tikzpicture}[scale = 1, baseline = 10ex]
\draw [color = black!60, thick] (0,1) -- (1,1);
\draw [color = black!60, thick] (1,0) -- (1,1);
\draw [thick] (.5,.5-.15) -- (1.5,.5-.15); 
\node [draw,fill=white,inner sep=1pt] (A) at (.5,.35) {$u_2 1$};
\node [draw,fill=white,inner sep=1pt] (B) at (1.5,.35) {$3$};
\node [draw,fill=white,inner sep=1pt] (C) at (.5,1.35) {$u_1 2$};
\node [scale = .8] (D) at (.5,.75) {$f_1$};
\node [scale = .8] (D) at (1.5,.75) {$f_1$};
\node [scale = .8] (D) at (.5,1.75) {$f_2$};
\draw [black, thick] (0,0) -- (2,0) -- (2,1) -- (1,1) -- (1,2) -- (0,2) -- cycle;
\end{tikzpicture} \\ 
u_2&+& u_2 u_1 & + & u_2 u_1^2& ~~~~~~~~ = u_2+1+u_1 = 0.
\end{align*}
\caption{An orbit of $\psi'$ when $\vec{\lambda} = ((2,1), (0),(0))$, $F=[f_1^2,f_2]$, and $p(f_i) = i$.}
\label{figure:orbit_psi_prime}
\end{figure}
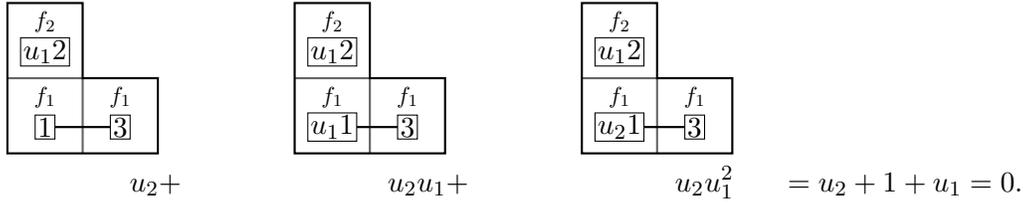

\subsection{The final fixed points}
Let $\mathcal{T}_{\vec{\gamma},F}$ be the fixed points of $\psi'$. 
For $P\in \mathcal{T}_{\vec{\gamma},F}$, let $S(P)$ be the tableau obtained by removing everything except for the colors. From Remark~\ref{remark:P_prime_semistandard},
since $\mathcal{T}_{\vec{\gamma},F} \subseteq \mathcal{P}_{\vec{\gamma},F}'$, we must have that $S(P)$ is a semistandard tableau in the colors of $F$. Since $P$ is fixed by $\psi'$, we must have that if the color $f_i$ appears in $\gamma^r$, then $p(f_i) = r \mod k$. All together, this means that 
$S(P) \in \SSYT_k(\vec{\gamma},F).$ We also have that $\weight(P)$ only depends on the colors, so $\weight(P) = \rho(P)$ as described in Theorem \ref{theorem:mainresult}.

Given $T \in \SSYT_k(\vec{\gamma},F)$, we now want to construct all elements $P \in \mathcal{T}_{\vec{\gamma},F}$ such that $S(P) =T$. But this can be done in a very simple way: 

Choose any $\sigma = u_{a_1} \sigma_1 \cdots u_{a_n} \sigma_n$ written in one-line notation. Read the cells of $T$ from top to bottom, left to right, and place $u_{a_1} \sigma_1$ in the first cell, $u_{a_2} \sigma_2$ in the second cell, and so on.
Now for each row in which the color $f_i$ appears consecutively, there is a unique way to connect this sequence of cells colored by $f_i$ with rim hooks so that we have a valid element of $\mathcal{P}_{\vec{\gamma},F}$. This is given precisely by the decreasing cycle order, as in Remark~\ref{remark:decreasing_cycle_bijection}, so that we only connect $c$ to its left if its index is not a left-to-right minimum in this sequence of cells. In the end, we find that
\begin{proposition}
For any $\vec{\gamma} \vdash_k n$ and color rule $F$, 
\[
\sum_{P \in \mathcal{T}_{\vec{\gamma},F}} \weight(P) = n! k^n \sum_{T\in \SSYT_k(\vec{\gamma},F)} \rho(F).
\]
\end{proposition}

Theorem \ref{theorem:mainresult} then follows.

\section{Acknowledgments}
The first author would like to thank Ilse Fischer and Bal\'{a}zs Szendr\H{o}i for their supervision and suggesting research directions. The first author would also like to acknowledge the support given by his employment at the University of Vienna.

The second author would like to thank, first and foremost, Jeffrey B. Remmel, who initially suggested applying the methods in \cite{MendesRomero} to the case in Subsection~\ref{section:wreath_defining_representation} in 2014.
It would have been great to share with him this refined approach, but we are grateful that he was able to see the results in Subsection~\ref{section:wreath_defining_representation} around that time.
The second author was supported by ERC grant “Refined invariants in combinatorics, low-dimensional topology
and geometry of moduli spaces,” No. 101001159.

Both authors would like to thank the organizers of the Alpine Algebraic Geometry Workshop at Obergurgl, 2024, where this collaboration started. We would also like to thank Bal\'{a}zs Szendr\H{o}i for his many helpful comments and for hinting to us the statement and proof of point 2 in Proposition \ref{proposition:regular-sequence}.

\bibliographystyle{amsalpha}
\bibliography{main}

\end{document}